\documentclass[11pt]{amsart}
\usepackage{latexsym, amsmath, amssymb,amsthm,amsopn,amsfonts}
\usepackage{version}
\usepackage{epsfig,graphics,color,graphicx,graphpap}
\usepackage{amssymb}
\usepackage{comment}
\usepackage[normalem]{ulem}
\usepackage{xcolor} 
\newcommand{\LC}{\left(}
\newcommand{\RC}{\right)}

\newcommand{\p}{\partial}

\newcommand{\F}{\mathcal{F}}
\newcommand{\HH}{\mathcal{H}}
\newcommand{\LL}{\mathcal{L}}

\newcommand{\W}{\mathcal{W}}
\newcommand{\R}{\mathbb R}
\newcommand{\N}{\mathbb N}
\newcommand{\D}{\Delta}
\DeclareMathOperator{\supp}{supp}
\DeclareMathOperator{\argmin}{argmin}
\DeclareMathOperator{\Dom}{Dom}

\numberwithin{equation}{section}
\newtheorem{theorem}{Theorem}[section]

\newtheorem{proposition}[theorem]{Proposition}
\newtheorem{lemma}[theorem]{Lemma}
\newtheorem{definition}{Definition}[section]

\newtheorem{remark}{Remark}[section]
 
\newcommand\redsout{\bgroup\markoverwith{\textcolor{red}{\rule[0.5ex]{2pt}{0.4pt}}}\ULon}

\author[R.-Y. Lai]{Ru-Yu Lai}
\address{School of Mathematics, University of Minnesota, Minneapolis, MN 55455, USA}
\curraddr{}
\email{rylai@umn.edu }

\author[Y.-H. Lin]{Yi-Hsuan Lin}
\address{Department of Mathematics and Statistics, University of Jyv\"askyl\"a, 40014 Jyv\"askyl\"a, Finland}
\curraddr{}
\email{yihsuanlin3@gmail.com}

\author[A. R\"uland]{Angkana R\"uland}
\address{Max-Planck Institute for Mathematics in the Sciences, Inselstraße 22, 04103 Leipzig, Germany}
\curraddr{}
\email{rueland@mis.mpg.de}

\title[A global uniqueness for fractional parabolic equations]{The Calder\'on problem for a space-time fractional parabolic equation}
 
\begin{document}
\maketitle
\begin{abstract} 
In this article we study an inverse problem for the space-time fractional parabolic operator $(\p_t-\Delta)^s+Q$ with $0<s<1$ in any space dimension. We uniquely determine the unknown bounded potential $Q$ from infinitely many exterior Dirichlet-to-Neumann type measurements. This relies on Runge approximation and the dual global weak unique continuation properties of the equation under consideration. In discussing weak unique continuation of our operator, a main feature of our argument relies on a Carleman estimate for the associated fractional parabolic Caffarelli-Silvestre extension.
Furthermore, we also discuss constructive single measurement results based on the approximation and unique continuation properties of the equation.
	
	\medskip
	
	
	%\noindent{\bf Mathematics Subject Classification (2010)}:
	
\end{abstract}

\tableofcontents

\section{Introduction}
In this article we consider global uniqueness properties for a nonlocal inverse problem for a \emph{space-time} fractional parabolic equation. More precisely, let $T>0$ be a real number and $\Omega$ be a bounded and open set in $\R^n$ for $n\geq 1$, and $Q=Q(t,x)\in L^\infty((-T,T)\times \Omega)$.
Suppose that $u=u(t,x)$ satisfies the following initial exterior value problem:  
\begin{align}\label{noncal_heat}
\begin{cases}
\left( (\p_t-\Delta)^s +Q(t,x)\right)u(t,x)=0 & \hbox{in }\Omega_T:=(-T,T)\times \Omega,\\
u(t,x)=f(t,x) & \hbox{in }(\Omega_e)_T:=(-T,T)\times (\R^n \setminus \overline{\Omega}),\\
u(t,x)=0 & \hbox{for } (t,x)\in (-\infty,-T]\times \R^n,
\end{cases}
\end{align}
where $\Delta=\Delta_x$ and $\Omega_e=\R^n \setminus \overline{\Omega}$.
Then we seek to recover information on $Q$ from the exterior Dirichlet-to-Neumann data
\begin{align}\label{DN map0}
\Lambda_Q: u|_{(\Omega_e)_T}\mapsto (\p_t-\Delta)^su|_{(\Omega_e)_T},
\end{align}
where $u$ is the solution of \eqref{noncal_heat}.
Here $\eqref{DN map0}$ is to be viewed as the nonlocal analogue of the classical Dirichlet-to-Neumann map.
As one of our main results, we prove that it is possible to recover $Q$ in $\Omega_T$ from the measurements of this Dirichlet-to-Neumann map.

Our problem combines two main features by involving a \emph{nonlocal} and \emph{non-self adjoint} operator mixing space and time. This is reflected in our analysis of the problem. Let us discuss these two aspects:
\begin{itemize}
\item \textbf{Non-self adjointness.}
Our inverse problem can be regarded as a nonlocal version of the Calder\'on problem for the local parabolic operator $$\LL:=\p_t-\Delta.$$
In contrast to the fractional Laplacian, both the operator $\LL$ and our nonlocal realization of it, $\LL^s = (\p_t - \Delta)^s$, are \emph{not} self-adjoint operators. As in the case of the Calder\'on problem for the standard heat equation, this is reflected in the central Alessandrini type identities on which our uniqueness arguments rely. 
\item \textbf{Nonlocality in space and time.}
As in the fractional Calder\'on problem, on which there is an increasing amount of literature, the fractional space-time problem studied here involves a nonlocal operator. Comparing the degrees of freedom involved vs the given information, this implies that the inverse problem under consideration is (formally) determined.
As in the fractional Calder\'on problem, this implies that one can not only hope for ``infinite data'' but also for ``single measurement'' uniqueness results.
In the sequel, we prove that indeed both of these hopes are justified:
Exploiting strong approximation results, we first show the global uniqueness in the fractional parabolic equation with (partial) exterior measurements. In particular, given the nonlocal Dirichlet-to-Neumann (DN) map which formally reads
$
\Lambda_Q: u|_{(\Omega_e)_T}\mapsto (\p_t-\Delta)^su|_{(\Omega_e)_T}
$
and $u$ is the solution of \eqref{noncal_heat}, one can recover the unknown bounded potential $Q$ inside $\Omega_T$.
Its rigorous mathematical formulation and the related analysis for the DN map will be discussed in Section \ref{Section 3}. Secondly, following the ideas from \cite{GRSU18}, we further prove a single measurement uniqueness result.

The fractional parabolic equation \eqref{noncal_heat} contains nonlocally coupled space-time derivatives. This special feature distinguishes \eqref{noncal_heat} from equations like $\p_tu+(-\Delta)^s u=f$ and $\p_t^\alpha u+(-\Delta)^su=f$ with $0<\alpha,s<1$, where space and time are ``decoupled''.
In particular, the time variable $t$ in the operator $(\p_t-\Delta)^s$ acts like an additional direction of the space. This has direct  implications on our analysis of the problem: Unlike the local parabolic equation whose well-posedness can be carried out by standard Galerkin formulation,  the operator $(\p_t-\Delta)^s$ possesses certain features of elliptic operators provided that suitable spaces and norms are chosen. This makes it possible to regard the time variable $t$ as an additional direction in space and allows us to formulate the problem as a problem that shares features with elliptic equations. Therefore, with the application of Lax-Milgram arguments, the well-posedness of the fractional parabolic problem holds in a suitable function space, which will be discussed in Section \ref{well-posedness_statement}.
\end{itemize}

We further remark, that the nonlocality of the problem in combination with the non-self adjointness also leads to a ``memory'' (in time) of the equation which could be of interest in certain applications.

\subsection{Literature in nonlocal inverse problems}
Let us comment on the literature related to our problem: The Calder\'on problem is a mathematical model of electrical impedance tomography (EIT) and has been studied intensively, involving various aspects including uniqueness, stability estimates, reconstructions and numerical implementation. For its detailed development, we refer to the survey paper \cite{Gunthersurvey} and the references therein. 

Also, substantial effort has been devoted to the study of its local parabolic version (i.e. for \eqref{noncal_heat} with $s=1$).
In particular, the global uniqueness for the local parabolic equation was studied by \cite{canuto2001determining} for the linear case, and by \cite{klibanov2004global} for the nonlinear case. Moreover, stability properties have been studied in  \cite{choulli2012stability,choulli2006some,gaitan2013stability}. We also refer to a survey article for inverse problems for anomalous diffusion \cite{jin2015tutorial}.

The inverse problem for the nonlocal operator $\LL^s+Q$ however differs rather strongly from its local analogue or the time fractional diffusion equation (see \cite{kian2018global}). It rather resembles properties of the fractional Calder\'on problem whose study had been initiated in \cite{ghosh2016calder} where the authors showed that an unknown potential $q=q(x)$ can be uniquely determined from the DN map, given by
$$
\Lambda_q: u|_{\Omega_e}\mapsto (-\Delta)^su|_{\Omega_e},
$$
for infinitely many (partial data) measurements. Here the function $u=u(x)$ is the unique solution to the fractional Schr\"odinger equation 
\begin{align}\label{noncal_elliptic}
\left\{\begin{array}{ll}
\left( (-\Delta)^s +q\right)u =0 & \hbox{in }\Omega ,\\
u =f & \hbox{in }\Omega_e.
\end{array}\right.
\end{align}
This global uniqueness result in \cite{ghosh2016calder} has been extended to a single measurement in \cite{GRSU18} and to low regularity potentials in \cite{RS17a}. Moreover, it was shown that the problem is logarithmically stable \cite{RS17a} and that this is optimal \cite{RS17c}.  
Closely related (partial data) uniqueness results were studied for different types of fractional equations including the anisotropic nonlocal elliptic problem \cite{ghosh2017calderon} and the semilinear equation which was was studied in \cite{LL17}. Further recent  developments include the study of the fractional Calder\'on problem with lower order contributions \cite{bhattacharyya2018inverse,cekic2018calder} (non-self adjoint problems), reconstruction algorithms based on monotonicity tests \cite{HL17,harrach2019monotonicity2}, the recovery of embedded obstacles \cite{CLL17}, the study of the fractional Helmholtz systems \cite{cao2018determining} and the quantitative approximation properties of nonlocal operators \cite{ruland2017quantitative,rulandsalo2017quantitative}. For further information, we refer to the survey article \cite{Sa17} on inverse problems for fractional elliptic operators.

\subsection{Background and applications}
The derivation of the limiting distribution of an ensemble of particles following a specified stochastic process provides a way to develop physical models for anomalous diffusion. 
The continuous time random walk (CTRW) \cite{BMM2005, MBenson, MWeiss, ScherLax} can be used to determine these limits when the particles' jumps have infinite variance, or the waiting times between the particles' jumps have infinite mean. In particular, large particle jumps and long waiting times are associated with fractional derivatives in space or in time, respectively. When the jump sizes and waiting times are independent, the governing equation generated by the limit process contain fractional order spatial or temporal derivatives. 

In the CTRW the size of the particle jumps can depend on the waiting time between jumps, that is, the jumps and waiting times are coupled. This results in a different kind of limiting particle distribution governed by a fractional differential equation involving coupled space-time fractional derivative operators \cite{BMS2004, MBenson}. These coupled CTRW have been studied in a variety of physical systems \cite{Metzler} and have been used to describe transport in chaotic and turbulent flows in hydrodynamics \cite{SWK1987, SWS1993}. For a detailed discussion and applications, we refer to \cite{BMM2005, KRSbook, MBenson} and the references therein. We note that
the fractional heat equation \eqref{noncal_heat} is the governing equation whose random jumps coupled with the random waiting times, while $\p_t^\alpha u+(-\Delta)^su=f$ describes the jumps and the waiting times are dependent.

%%%%%%%%%%%%%%%%%  Main Results  %%%%%%%%%%%%%%%%%%%%%
\subsection{Main results}

Let us describe our main results:
For $0<s<1$, the fractional heat operator $\mathcal{L}^su$ of a function $u=u(t,x):\R^{n+1}\rightarrow \R$, $n\geq 1$, is given by 
\begin{align}
\label{eq:op}
\widehat{\LL^su}(\rho,\xi)=(i\rho+|\xi|^2)^s\widehat{u}(\rho,\xi), \quad \text{ for }u\in \mathcal{S}(\R^{n+1}),
\end{align}
where $\mathcal S( \R^{n+1})$ denotes the Schwartz space of rapidly decreasing functions, for $\rho\in \R$ and $\xi\in\R^n$.

We assume that $0$ is not a Dirichlet eigenvalue of \eqref{noncal_heat}, i.e.,
\begin{align}\label{eigenvalue condition} 
\begin{cases}
\hbox{If }u\in\HH^s(\R^{n+1})\ \hbox{solves }(\LL^s+Q)u=0\ \hbox{in }\Omega_T \\
\hbox{with $u|_{(\Omega_{e})_T}=0$} \text{ and }u=0 \text{ for }(t,x)\in (-\infty, -T]\times \R^n, \\
 \hbox{then }u\equiv 0 \text{ in }(-\infty,T)\times \R^n.
\end{cases}
\end{align} 
Notice that the condition \eqref{eigenvalue condition} only ensures the solution $u\equiv 0$ in $\Omega_T$ (since the initial value and the boundary data are zero there) while not necessarily imposing conditions on the future behavior of the solution. The function space $\HH^s(\R^{n+1})$ will be explained in detail in Section \ref{Section 2}. As a consequence of \eqref{eigenvalue condition} we will deduce well-posedness of the corresponding Dirichlet problem.
Thus, for example, when $Q(t,x)\geq 0$ for $(t,x)\in (-T,T]\times \Omega$, the exterior value problem \eqref{noncal_heat} is well-posed (see Section \ref{Section 3}).
Given the assumption \eqref{eigenvalue condition}, the associated (parabolic) \emph{Dirichlet-to-Neumann} (DN) map is can be formally defined by
\begin{align}\label{DN map}
\Lambda_Q :\mathbb X \rightarrow \mathbb X^*\quad \text{ with }\quad 
\Lambda_Q: f\mapsto \LL^su|_{(\Omega_{e})_T},
\end{align}
where function spaces $\mathbb X$, $\mathbb X^*$ will be explained precisely in Section \ref{Section 3}.

With this notation in hand, it is possible to present our first main result:

\begin{theorem}[Global uniqueness]\label{MAIN THEOREM}
	Let $n\geq 1$, $\Omega\subset \R^n$ be a bounded open set and $T\in (0,\infty)$ be a real number. Let $Q_1(t,x)$ and $Q_2(t,x)$ be two bounded potentials satisfying the eigenvalue condition \eqref{eigenvalue condition}. Suppose that $\mathcal U_1,\mathcal U_2$ are arbitrary open sets in $\R^n\setminus \overline{\Omega}$ and $\Lambda_{Q_j}$ is the DN map with respect to $(\LL^s+Q_j)u=0$ in $\Omega_T$ for $j=1,2$. If 
	\begin{align*}
	\Lambda_{Q_1}f|_{(-T,T)\times\mathcal U_2}=\Lambda_{Q_2}f|_{(-T,T)\times\mathcal U_2}\text{ for all }f=f(t,x)\in C^\infty_c((-T,T)\times\mathcal U_1),
	\end{align*}
	then it holds
	\begin{align*}
	Q_1=Q_2 \text{ in }\Omega_T.
	\end{align*}
\end{theorem}
Note that, to determine the potential, we only utilize the exterior partial measurements as described in Theorem \ref{MAIN THEOREM}. This can be regarded as a space-time nonlocal parabolic inverse problem with partial data information. 

As in the elliptic fractional Calder\'on problem, an important ingredient in the derivation of the result are strong Runge approximation properties:

\begin{theorem}[Runge approximation]\label{Thm Runge approx.}
	For $n\geq 1$ and $T>0$, let $\Omega\subset \R^n$ be a bounded open set and $\widetilde \Omega$ be an open set containing $\Omega$ satisfying $\Omega\Subset \widetilde \Omega$. If $Q(t,x)\in L^\infty(\Omega_T)$ satisfies \eqref{eigenvalue condition}, then for any $g\in L^2(\Omega_T)$ and for any $\epsilon>0$, there exists a solution $u_\epsilon\in \HH^s(\R^{n+1})$ which solves
	\begin{align*}
	(\LL^s+Q)u_\varepsilon=0 \text{ in }\Omega_T \text{ with }\mathrm{supp}(u_\epsilon )\subset  \widetilde \Omega_T:=(-T,T)\times\widetilde{\Omega},
	\end{align*}
	and satisfies
	\begin{align*}
	\|u_\epsilon -g\|_{L^2(\Omega_T)}<\epsilon.
	\end{align*} 
\end{theorem}

We remark that in the works \cite{dipierro2016local,ruland2017quantitative, CDV18, CDV18a}, the authors studied the approximation property for a different class of operators, which involves combinations of both local and nonlocal operators.
Similarly as in the elliptic setting, the derivation of our Runge approximation results relies on a \emph{global weak unique continuation property} of the nonlocal equation:

\begin{theorem}[Global weak unique continuation]\label{Thm UCP}
	For $n\geq 1$, let $T>0$ be a real number and $\mathcal U\subset \R^{n}$ be an arbitrary nonempty open set.  Let $u\in \HH^s({ \R^{n+1}})$ satisfy $u=\LL^su=0$ in $(-T,T)\times \mathcal U$. Then $u(t,x)\equiv 0 $ for $(t,x)\in  (-T,T)\times \R^n$.
\end{theorem}

\begin{remark}
We make the following observations:
\begin{itemize}
\item[1. ] A major difference between the nonlocal elliptic and the nonlocal parabolic cases is the effect of their strong uniqueness properties. Indeed, Theorem \ref{Thm UCP} displays a non-symmetric behaviour with respect to the spatial and temporal variables: While there is a strong nonlocal effect in the spacial $x$-variable, it is not possible to propagate the global vanishing in the future $t$-direction.

\item[2. ] In order to be applied in the context of our inverse problems, it is important to have a global weak unique continuation result at our disposal which still holds if the condition $\mathcal{L}^s u = 0 = u$ only holds \emph{locally} in space and time. While both strong and weak unique continuation results for the operator $\LL^s$ are available if the equation is assumed to hold \emph{globally} (see \cite{BG2017}), we here provide an alternative \emph{local} proof of the global weak unique continuation property based on an appropriate Carleman estimate.

\item[3. ] When the function $u(t,x)\equiv u(x)$ is a time-independent function, then by virtue of \cite[Corollary 1.4]{stinga2017regularity}, the fractional space-time operator $\LL^s$ simplifies to the fractional Laplacian: $\LL^s u = (-\Delta )^su$. For this operator the analogue of the global weak unique continuation property of Theorem \ref{Thm UCP} has been proved in \cite[Theorem 1.2]{ghosh2016calder}, see also \cite{DSV17}. 
\end{itemize}
\end{remark}

Last but not least, similarly as in \cite{GRSU18}, the global weak unique continuation properties of the equation at hand, also allow us to constructively recover $Q\in C^0((-T,T)\times \Omega)$.

\begin{theorem}
\label{thm:single_measurement}
Let $s\in (0,1)$, $T>0$ and $Q \in C^0((-T,T)\times \Omega)$. Let $W \subset \Omega_e$ such that $\overline{W}\cap \overline{\Omega} = \emptyset$. Let $f\in \widetilde{\mathcal{H}}^s((-T,T)\times W) \setminus \{0\}$. Then there is a reconstructive method to uniquely recover the potential $Q$ from $f$ and $\Lambda_Q(f)$.
\end{theorem}

As in \cite{GRSU18} the key to this single measurement result is the formal determinedness of the inverse problem under consideration in combination with the global weak unique continuation results.

\subsection{Outline}
The paper is structured as follows. In Section \ref{Section 2}, we review basic properties for the nonlocal parabolic operator $\LL^s$ and introduce the related function spaces. The well-posedness of the Dirichlet problem \eqref{noncal_elliptic} and the associated Dirichlet-to-Neumann map $\Lambda_Q$ will be discussed in Section \ref{Section 3}. In Section \ref{Section 4}, 
we study the extension problem for the nonlocal parabolic equation. 
Relying on the properties of the parabolic Caffarelli-Silvestre extension, we discuss the global weak unique continuation property in Section \ref{Section 5}. Here we also derive a suitable Carleman estimate for the equation under consideration.
Finally, the Runge approximation property and all remaining main results of the article will be deduced in Section \ref{Section 6}.

%%%%%%%%%%%    Preliminary      %%%%%%%%%%%%%%%%%%%%%%%%%%%%%%%%%%%%%
\section{Preliminary results}\label{Section 2}
We start by defining the space-time fractional parabolic operator $\LL^s=(\partial_t -\Delta)^s$ for $0<s<1$ and recalling some of its properties.

%%%%%%%%%%%   Definitions   %%%%%%%%%%%%%%%%%%%%%%%%%%%%%%%
\subsection{The fractional parabolic operator}
For $0<s<1$, the nonlocal operator $\LL^s$ is defined through its Fourier representation
$$
\widehat{\LL^s u}(\rho,\xi)= (i\rho+|\xi|^2)^s\widehat{u}(\rho,\xi), \quad \text{ for }u\in \mathcal{S}(\R^{n+1}),
$$
where $\mathcal S( \R^{n+1})$ denotes the Schwartz space of rapidly decreasing functions. 
Here $\widehat{u}$ is the Fourier transform of $u$ defined by 
$$
\widehat{u}(\rho,\xi)=[\mathcal{F}_t(\mathcal F_x u)](\rho,\xi)=\int_{\R^{n+1}} e^{-i(t,x)\cdot (\rho,\xi)} u(t,x)dtdx,\text{ for }\rho \in \R\text{ and }\xi \in \R^n,
$$
where $\mathcal F_t$ and $\mathcal F_x$ are the Fourier transformations in $t\in \R$ and $x\in \R^n$, respectively.
As a function space which is naturally associated with this operator we consider the $L^2$-based function space
$$
\text{Dom}(\LL^s) := \{u\in L^2(\R^{n+1}):\ (i\rho+|\xi|^2)^{s} \widehat{u}(\rho,\xi) \in L^2(\R^{n+1})\}.
$$
We note that then in particular $\text{Dom}(\LL^s)\subset L^2(\R^{n+1})$.

\begin{remark}
\label{rmk:semi_group}
We remark that in addition to the described Fourier point of view, it is also possible to adopt a semi-group perspective in defining this operator. 
From this, one obtains the representation
$$
    \LL^su(t,x) = {1\over \Gamma(-s) }\int^\infty_0 \LC e^{-\tau \LL} u(t,x)-u(t,x)\RC {d\tau\over \tau^{1+s}}\ \ \hbox{in }L^2(\R^{n+1}), \ u \in \Dom(\LL^s),
$$
which can also be reformulated in terms of an integral representation
\begin{align}
\label{eq:integral_rep}
\LL^s u(t,x) = \int^\infty_0 \int_{\R^n} (u(t,x)-u(t-\tau,x-z)) K_s(\tau,z) dzd\tau, \ u \in \Dom(\LL^s),
\end{align}
where 
	\begin{align*}
	K_s(\tau,z) = {1\over (4\pi)^{n/2} |\Gamma(-s)|} {e^{-|z|^2/(4\tau)} \over \tau^{n/2+1+s}},\ \ \tau>0,\ \ z\in\R^n
	\end{align*}
	is the kernel of $\LL^s$.
We refer to \cite{stinga2017regularity} for a more detailed discussion on this. 
\end{remark}

\subsection{Function spaces}\label{function space}
Given an open set $O\subset \R^{n+1}$, if $f=f(t,x)$ and $g=g(t,x)$ are $L^2$ functions in $O$, we denote the $L^2$ inner product by 
\begin{align*}
( f,g )_O:=\int_{O} f\overline{g}\ dxdt.
\end{align*}
For the nonlocal space-time operator $\LL^s=(\partial_t-\Delta)^s$, we define the following associated function spaces, which are slightly different from the usual fractional Sobolev spaces. For any $a\in \R$, we set 
\begin{align*}
\HH^a (\R^{n+1}):=\left\{u\in L^2(\R^{n+1}): \|u\|_{\HH^a(\R^{n+1})}<\infty \right\},
\end{align*}
where 
$$
\|u\|_{\HH^{a}(\R^{n+1})}^2 = \int_{\R^{n+1}} (1+|i\rho+|\xi|^2|)^{a} |\widehat{u}(\rho,\xi)|^2 d\rho d\xi <\infty.
$$
Note that $|i\rho +|\xi|^2|=\left(|\rho|^2+|\xi|^4\right)^{1/2}$ and $2^{-1/2}(|\rho|+|\xi|^2)\leq (|\rho|^2+|\xi|^4)^{1/2}\leq |\rho|+|\xi|^2$. It is not hard to see that $ \text{Dom}(\LL^s)= \HH^{2s}(\R^{n+1})$ for $0<s<1$.
In addition, for an open set $O$ in $\R^{n+1}$, $n\geq 1$, we use the following notation
\begin{align*}
\HH^a(O)&:= \left\{u|_O:\ u\in \HH^a(\R^{n+1})\right\},\\
\widetilde{\HH}^a(O)&:= \hbox{closure of $C^\infty_c(O)$ in $\HH^a(\R^{n+1})$}.  
\end{align*} 
Note that $C^\infty_c(\R^{n+1})$ is dense in $\HH^a(\R^{n+1})$ under the norm $\|\cdot \|_{\HH^a(\R^{n+1})}$. 
We also observe that
$$
(\HH^a(O))^* = \widetilde{\HH}^{-a}(O),\ (\widetilde{\HH}^a(O))^* = \HH^{-a}(O),\text{ for } a\in\R.
$$
Let $E\subset \R^{n+1}$ be a closed set, then we define 
$$
\HH^s_E = \HH^s_E (\R^{n+1})=\left\{u\in \HH^s(\R^{n+1}): \ \mathrm{supp}(u)\subset E \right\}.
$$ 
These properties follow in the same way as the ones for the ``classical'' fractional Sobolev spaces for which we refer to \cite{mclean2000strongly}.

Additionally, we recall the (standard) fractional Sobolev spaces (where the space and time variables are separated). 
Let $a\in \R$ be a constant and $H^{a}(\mathbb{R}^{n})=W^{a,2}(\mathbb{R}^{n})$
be the $L^2$-based fractional Sobolev space (see \cite{di2012hitchhiks} for example) with the norm 
\[
\|u\|_{H^{a}(\R^{n})}:=\left\Vert \mathcal{F}_x^{-1}\big\{\left\langle \xi\right\rangle ^{a}\mathcal{F}_xu\big\}\right\Vert _{L^{2}(\R^{n})},
\]
where $\left\langle \xi\right\rangle =(1+|\xi|^{2})^{\frac{1}{2}}$.
Let $\mathcal O \subset \R^n$ be an open set, then 
\begin{align*}
H^{a}(\mathcal{O}) & :=\{u|_{\mathcal{O}}:\,u\in H^{a}(\mathbb{R}^{n})\},\\
\widetilde{H}^{a}(\mathcal{O}) & :=\text{closure of \ensuremath{C_{c}^{\infty}(\mathcal{O})} in \ensuremath{H^{a}(\mathbb{R}^{n})}},
\end{align*}
and the Sobolev space $H^{a}(\mathcal{O})$ is complete under the norm
\[
\|u\|_{H^{a}(\mathcal{O})}:=\inf\left\{ \|v\|_{H^{a}(\mathbb{R}^{n})}:\ v\in H^{a}(\mathbb{R}^{n})\mbox{ and }v|_{\mathcal{O}}=u\right\} .
\]

\subsection{Mapping Properties of the operator $\LL^s$} 
Using the definition of $\LL^s$, we note that the following mapping property holds:

\begin{lemma}\label{Lemma 1}
Let $b\geq 0$, $a\in \R$ be constants. Then the fractional heat operator is a bounded map 
	$$
	    \LL^b : \HH^{2a}(\R^{n+1})\rightarrow \HH^{2a-2b}(\R^{n+1}).
	$$
\end{lemma}
\begin{proof}
If $u\in \HH^{2a}(\R^{n+1})$, then 
\begin{align*}
\left\|\LL^b u\right\|_{\HH^{2a-2b}(\R^{n+1})}
&=\left\|(1+|i\rho+|\xi|^2|)^{a-b} |\widehat{\LL^b u}(\rho,\xi)|\right\|_{L^2(\R^{n+1})} \\
&=\left\|(1+|i\rho+|\xi|^2|)^{a-b} |i\rho+|\xi|^2|^b|\widehat{u}(\rho,\xi)|\right\|_{L^2(\R^{n+1})} \\
&\leq C\left\| (1+|i\rho+|\xi|^2|) ^{a }  |\widehat{u}(\rho,\xi)|\right\|_{L^2(\R^{n+1})} \\
&=C\|u\|_{\HH^{2a}(\R^{n+1})}.
\end{align*} 
\end{proof}
In the remainder of this paper, we will only consider the case $a=s/2$ and $b=s$, that is, $\LL^s: \HH^{s }(\R^{n+1})\rightarrow \HH^{-s }(\R^{n+1})$ for $0<s<1$.

Before addressing the well-posedness for our space-time Dirichlet problem, we discuss the adjoint of the operator $\LL^s$:

\begin{lemma}
Let $0<s<1$ and define $\LL^s_{\ast}$ by
\begin{align*}
\widehat{\LL^s_{\ast} w}(\rho,\xi)
= (- i \rho + |\xi|^2)^s \widehat{w}(\rho,\xi), \quad \mbox{ for } w\in \Dom(\LL^s).
\end{align*} 
Then, for $u,w\in \HH^s(\R^{n+1})$, one has
\begin{align}\label{identity}
\left\langle \LL^s  u, w  \right\rangle_{\HH^{-s}\times \HH^s} = \left\langle u, \LL^s _*w\right\rangle_{\HH^{s}\times \HH^{-s}}\text{ and }%\ \ \hbox{for $w\in \HH^s(\R^{n+1})$}
\end{align}
\begin{align}\label{2s}
\left\langle \LL^{s/2}(\LL^{s/2}u),  w\right\rangle_{\HH^{-s}\times \HH^s} = \left\langle \LL^{s }u, w \right\rangle_{\HH^{-s}\times \HH^s}.
\end{align}
Furthermore, it also holds that
\begin{align}\label{integration by parts formula}
\langle \LL^{s} u ,w \rangle _{\HH^{-s}\times \HH^s}=\left(
\LL^{s/2}u , \LL^{s/2}_* w\right)_{\R^{n+1}}.
\end{align}
\end{lemma}

\begin{remark}
We remark that especially the identity \eqref{integration by parts formula} will play an essential role in obtaining the well-posedness for the fractional parabolic equation \eqref{noncal_heat} in the following section.
\end{remark}

\begin{proof}
By utilizing the Plancherel theorem, for $u,w \in \mathcal S(\R^{n+1})$, we have 
\begin{align*}
\langle \LL^s u, w\rangle _{\HH^{-s}\times \HH^s}=\int_{\R^{n+1}}(\LL^s u) \overline{w}dxdt=\int_{\R^{n+1}}(\widehat{\LL^s u}) \overline{\widehat w}d\xi d\rho=\int_{\R^{n+1}} (i\rho +|\xi|^2)^s\widehat{u} \overline{\widehat{w}}\ d\xi d\rho.
\end{align*}
Similarly, we also note that 
\begin{align*}
\langle  u, \LL^s_* w\rangle _{\HH^{s}\times \HH^{-s}}=\int_{\R^{n+1}} u (\overline{\LL^s_*w})dxdt=\int_{\R^{n+1}} \widehat u (\overline{\widehat{\LL^s_*w}})d\xi d\rho=\int_{\R^{n+1}} \widehat{u}\overline{ (-i\rho +|\xi|^2)^s\widehat{w}}\ d\xi d\rho,
\end{align*}
which proves \eqref{identity} for $u,w \in \mathcal{S}(\R^{n+1})$. Fourier transforming the respective identities, it is easy to see that \eqref{2s} holds whenever $u,w\in \mathcal S(\R^{n+1})$. In order to show \eqref{integration by parts formula} for $u,w \in \mathcal{S}(\R^{n+1})$, we only need to compute  
\begin{align*}
\left(
\LL^{s/2}u , \LL^{s/2}_* w\right)_{\R^{n+1}}=&\int_{\R^{n+1}}\LL^{s/2}u(\overline{\LL^{s/2}_* w})dxdt=\int_{\R^{n+1}}\widehat{\LL^{s/2}u}(\overline{\widehat{\LL^{s/2}_* w}})d\xi d\rho\\
=&\int_{\R^{n+1}}(i\rho +|\xi|^2)^{s/2}\widehat u [\overline{(-i\rho +|\xi|^2)^{s/2}\widehat w}]d\xi d\rho,
\end{align*}
by using the Plancherel theorem again, which proves \eqref{integration by parts formula} for the case $u,w\in \mathcal S(\R^{n+1})$. Therefore, \eqref{identity}, \eqref{2s} and \eqref{integration by parts formula} hold by using density arguments.
\end{proof}

In the sequel, we seek to study a ``time localized'' problem. To this end, we study the interaction of cut-off functions in the $t$-variable with our function spaces. 
Recall that a characteristic function on a Lipschitz domain is a multiplier of the Sobolev spaces $H^{\gamma}$ for $|\gamma |<\frac{1}{2}$. In the following observation, we also note that a characteristic function in the $t$ variable is a multiplier of $\HH^s(\R^{n+1})$, $s\in(0,1)$.

\begin{proposition}
\label{prop:cut-off}
	Let $\chi_{[-T,T]}(t)$ be a characteristic function for  $t\in \mathbb{R}$. Suppose that $u=u(t,x)\in \mathcal{H}^s(\mathbb{R}^{n+1})$, that is, 
	$$
	\int_{\mathbb{R}^n}\int_{\mathbb{R}} (1+|i\rho+|\xi|^2|)^s |\mathcal{F}u(\rho,\xi)|^2 d\rho d\xi<\infty,
	$$
	where $\mathcal{F}$ is the Fourier transform with respect to the $(t,x)\in \mathbb{R}^{n+1}$ variables.
	Then, the following observations hold: 
	\begin{itemize}
		\item[(1)] For each fixed space variable, we have
		$$
		\chi_{[-T,T]}(\cdot)u(\cdot,x)\in H^{s/2}(\mathbb{R})\ \ \hbox{for }a.e.\ x\in\mathbb{R}^n.
		$$
		
		\item[(2)] As a joint function of space and time, we have 
		$$\chi_{[-T,T]}u \in \mathcal{H}^s(\mathbb{R}^{n+1}).$$
	\end{itemize}
\end{proposition}

\begin{proof}
	(1) Recall that $\mathcal{F}_x$ is the Fourier transform in the $x$ variable and $\mathcal{F}_t$ is the Fourier transform in the $t$ variable, then in order to prove (1) for $u \in \mathcal{S}(\R^{n+1})$, we estimate
	\begin{align*}
	& \int_{\mathbb{R}^n} \int_{\mathbb{R}} (1+|\rho|)^s  | \mathcal{F}_tu(\rho,x)|^2  d\rho dx  
	=  \int_{\mathbb{R}} (1+|\rho|)^s \int_{\mathbb{R}^n} | \mathcal{F}_tu(\rho,x)|^2 dx d\rho \\
	=& \int_{\mathbb{R}} (1+|\rho|)^s\int_{\mathbb{R}^n}  |\mathcal{F}_x\left(\mathcal{F}_tu(\rho,\xi)\right)|^2 d\xi d\rho 
	=\int_{\mathbb{R}^n}\int_{\mathbb{R}} (1+|\rho|)^s |\mathcal{F}u(\rho,\xi )|^2 d\rho d\xi \\
	\leq &\int_{\mathbb{R}^n}\int_{\mathbb{R}} (1+|i\rho+|\xi|^2|)^s |\mathcal{F}u(\rho,\xi )|^2 d\rho d\xi <\infty,
	\end{align*}
	where we have utilized Plancherel's formula and Fubini's theorem. By density arguments this implies that for $u \in \mathcal{H}^s(\R^{n+1})$
	\begin{align*}
	\int_{\mathbb{R}} (1+|\rho|)^s  | \mathcal{F}_tu(\rho,x)|^2  d\rho <\infty \hbox{ for a.e.}\ x\in\mathbb{R}^n.
	\end{align*}
	Thus, we can deduce that 
	$$
	u(\cdot,x)\in H^{s/2}(\mathbb{R}) \mbox{ for a.e. $x \in \R^n$}.
	$$
	
	Next, we seek to show that for $0<s<1$, ($0<\frac{s}{2}<\frac{1}{2}$), the multiplication of a $\HH^s$ function $u( t, \cdot )$ with $\chi_{[-T,T]}$ is bounded, i.e.
	$$
	\chi_{[-T,T]}(\cdot)u(\cdot,x)\in H^{s/2}(\mathbb{R})\ \ \hbox{for }a.e.\ x\in\mathbb{R}^n.
	$$
	It suffices to consider the case $T=1$. Let us define the function 
	$$w(t,x):=\chi_{[-1,1]}(t)u(t,x) \ \ \ \hbox{in }\mathbb{R}^{n+1},$$ 
	then for each $x \in \R^n$ fixed and $u\in \mathcal{S}(\R^{n+1})$, we will show that
	\begin{align*}
	\int_{\mathbb{R}} (1+|\rho| )^s  | \mathcal{F}_t w(\rho,x)|^2  d\rho 
	\leq C \int_{\mathbb{R}} (1+|\rho| )^s  | \mathcal{F}_t u(\rho,x)|^2  d\rho .
	\end{align*}
	Recall that the Fourier transform of products turns into a convolution in the Fourier space, i.e.,
	\begin{align*}
	\int_{\mathbb{R}} (1+|\rho|)^s  | \left(\mathcal{F}_t w(\rho, x)\right)|^2  d\rho =\, &\int_{\mathbb{R}} (1+|\rho|)^s  \left| \left(\mathcal{F}_t \chi_{[-1,1]}* \mathcal{F}_tu\right)(\rho,x) \right|^2  d\rho\\
	\leq &\,C\int_{\mathbb{R}} (1+|\rho|)^s  \left|\int_{\mathbb{R}} \mathcal{F}_t{\chi}_{[-1,1]}(\eta)\mathcal{F}_t{u}(	\rho-\eta,x)d\eta \right|^2  d\rho,
	\end{align*}
	where for notational convenience here and in the sequel we have suppressed the space variable $\xi$. Furthermore,
	
	\begin{align*}
	&\quad \int_{\mathbb{R}} (1+|\rho|)^s \left|\int_{\mathbb{R}} \mathcal{F}_t{\chi}_{[-1,1]}(\eta) \mathcal{F}_t{u}(\rho-\eta,x)d\eta \right|^2d\rho\\
	&\leq c\int_{\mathbb{R}} (1+|\rho|)^s \sum_{k=1}^{\infty}\left|\int_{2\pi k\leq|\eta|\leq 2\pi(k+1)} {\sin(\eta)\over \eta }\mathcal{F}_t{u}(\rho-\eta,x)d\eta\right|^2d\rho\\
	&\quad +c\int_{\mathbb{R}} (1+|\rho|)^s \left|\int_{|\eta|\leq 2\pi} {\sin(\eta)\over \eta }\mathcal{F}_t{u}(\rho-\eta,x)d\eta\right|^2d\rho \\
	&=:I_1+I_2.
	\end{align*}
	We first estimate $I_1$. By the Cauchy-Schwartz inequality, one has 
	\begin{align*}
	I_1&=\int_{\mathbb{R}} (1+|\rho|)^s \sum_{k=1}^{\infty}\left|\int_{2\pi k\leq|\eta|\leq 2\pi(k+1)} {\sin(\eta)\over \eta }\mathcal{F}_t{u}(	\rho-\eta,x)d\eta\right|^2d\rho\\
	&\leq \int_{\mathbb{R}} \langle\rho\rangle^s \sum_{k=1}^{\infty}\left(\int_{2\pi k\leq|\eta|\leq 2\pi(k+1)} {1\over |\eta| }|\mathcal{F}_t{u}(\rho-\eta,x)|d\eta\right)^2d\rho \\
	&\leq C\int_{\mathbb{R}} \langle\rho\rangle^s \sum_{k=1}^{\infty} \int_{2\pi k\leq|\eta|\leq 2\pi(k+1)} {1\over |\eta|^2} |\mathcal{F}_t{u}(\rho-\eta,x)|^2 d\eta d\rho \\
	&\leq C \int_{\mathbb{R}}  \langle\rho\rangle^s \sum_{k=1}^{\infty} \int_{2\pi k\leq|\eta|\leq 2\pi(k+1)} {1\over \langle\eta\rangle^{2}\langle\rho-\eta\rangle^s}\langle\rho-\eta\rangle^s |\mathcal{F}_t{u}(\rho-\eta,x)|^2 d\eta d\rho \\
	&\leq C \int_{\mathbb{R}} {1\over \langle \eta\rangle^{2-s}} \left(\int_{\mathbb{R}}\langle\rho-\eta\rangle^s |\mathcal{F}_t{u}(\rho-\eta,x)|^2 d\rho \right) d\eta\\
	&= M\|u(\cdot,x)\|_{H^{s/2}(\mathbb{R})},
	\end{align*} 
	for some constant $M>0$. 
	Here we used that $|\eta|\geq 2\pi$ so that $|\eta|\geq C\langle\eta\rangle:=(1+|\eta|^2)^{1/2}$, and $$\langle\rho\rangle^s\langle\rho-\eta\rangle^{-s}\leq C\langle \eta\rangle^s \quad \text{ and }\quad \int_{\mathbb{R}} {1\over {\langle \eta \rangle}^{2-s}}d\eta \leq C<\infty, $$
	for some constant $C>0$, due to the fact that $2-s>1$ for $0<s<1$.  
	
	Secondly, we estimate $I_2$. By the Cauchy-Schwartz inequality again and the fact that $\frac{\sin^2(\eta)}{|\eta|^2} \leq 1$,
	we obtain that
	\begin{align*}
	I_2&\leq C\int_{\mathbb{R}} \langle\rho\rangle^s  \left|\int_{|\eta|\leq 2\pi} {\sin(\eta)\over \eta }\mathcal{F}_t{u}(	\rho-\eta,x)d\eta\right|^2d\rho \\
	&\leq C\int_{\mathbb{R}} \langle\rho\rangle^s \int_{|\eta|\leq 2\pi} {\sin^2(\eta)\over |\eta|^2\langle\rho-\eta\rangle^s }\langle\rho-\eta\rangle^s|\mathcal{F}_t{u}(\rho-\eta,x)|^2d\eta  d\rho\\ 
	&\leq C \int_{\mathbb{R}}\langle\rho\rangle^s \int_{|\eta|\leq 2\pi}{1\over\langle\rho-\eta\rangle^s}  \langle\rho-\eta\rangle^s|\mathcal{F}_t{u}(	\rho-\eta,x)|^2d\eta  d\rho\\
	&\leq C\int_{|\eta|\leq 2\pi} \langle \eta\rangle^s \left(\int_{\mathbb{R}}	\langle\rho-\eta\rangle^s|\mathcal{F}_t{u}(	\rho-\eta,x)|^2   d\rho\right) d\eta \\
	&\leq C\left(\int_{|\eta|\leq 2\pi} \langle \eta\rangle^s d\eta\right) \|u(\cdot,x)\|_{H^{s/2}(\mathbb{R} )}^2\\
	&	\leq C\|u( \cdot,x)\|_{H^{s/2}(\mathbb{R} )}^2,
	\end{align*}	
	where constants $C>0$ are independent of $u$, and they might changes from line to line.

	For (2), it suffices to show that
	\begin{align*}
	\int_{\mathbb{R}} (1+|\rho|+|\xi|^2)^s  | \mathcal{F}w(\rho,\xi)|^2  d\rho 
	\leq C \int_{\mathbb{R}} (1+|\rho|+|\xi|^2)^s  | \mathcal{F} u(\rho,\xi)|^2  d\rho .
	\end{align*}
Due to the observation 
	$$
	   C_{1,s} (1+|\rho|+|\xi|^2)^s\leq (1+|\rho|)^s+ (1+|\xi|^2)^s \leq C_{2,s}(1+|\rho|+|\xi|^2)^s ,
	$$
for some constants $0<C_{1,s}<1<C_{2,s}<\infty$,
	we only need to prove
	\begin{align*}
	\int_{\mathbb{R}} (1+|\rho|)^s  | \mathcal{F} w(\rho,\xi)|^2  d\rho 
	\leq C \int_{\mathbb{R}} (1+|\rho|)^s  | \mathcal{F} u(\rho,\xi)|^2  d\rho .
	\end{align*}
    In fact the above inequality holds by following a similar argument as part (1) by replacing $\mathcal{F}_t$ by $\mathcal{F}$ in those computations. 
		Finally, combined with a density argument, this completes the proof.
\end{proof}

\begin{remark}
\label{rmk:time_cut_off}
Note that by virtue of the integral representation (see Remark \ref{rmk:semi_group}) of the fractional parabolic operator $\mathcal{L}^s$, the value of $\mathcal{L}^s u(t,x)$ for $t\in (-\infty, T)$, is uniquely determined by the values of the function $u(t,x)$ for all $t \leq T$ and $x\in \mathbb{R}^n$. In other words, 
$$
(\mathcal{L}^s u)(t,x) = (\mathcal{L}^s (\chi_{(-\infty,T]}u)(t,x)) \text{ for }(t,x)\in \Omega_T,
$$ 
as by a straightforward computation for $u \in \mathcal{S}(\R^{n+1})$ and $(t,x)\in \Omega_T$,
\begin{align*} 
&\quad \mathcal{L}^su(t,x) \\
&= \int^{\infty}_0\int_{\mathbb{R}^n} \left[u(t,x)-u(t-\tau,x-z)\right]K_s(\tau ,z) dzd\tau\\
&= \int^{\infty}_0\int_{\mathbb{R}^n} \left[(\chi_{(-\infty,T]}u)(t,x)-(\chi_{(-\infty,T]}u))(t-\tau,x-z)\right]K_s(\tau ,z) dzd\tau\\
&=\mathcal{L}^s (\chi_{(-\infty,T]}u)(t,x).
\end{align*} 
Hence, in particular, in spite of the space-time nonlocal definition of the operator $\mathcal{L}^s$ in \eqref{eq:op}, the function $\LL^s u$ only depends on the past but \emph{not} on the future (as is physically desirable).

For the adjoint of the space-time nonlocal operator $\LL^s_*$ a similar computation yields that 
$$
(\mathcal{L}^s_* u)(t,x) = (\mathcal{L}^s_\ast (\chi_{[-T,\infty)}u)(t,x)) \text{ for }(t,x)\in \Omega_T.
$$ 
This means the function $\LL^s_* u$ only depends on the future but not on the past (symmetric property with respect to the time-variable of the adjoint operator $\LL^s$).

\end{remark}

Similar arguments hold for general characteristic functions $\chi_{(a,b)}(t)$ for $-\infty \leq a<b<\infty$ by scaling and translation with respect to time variables.

\section{Analysis of the initial exterior value problem for $\LL^s+Q$}\label{Section 3}
In this section, we study the well-posedness of the forward problem \eqref{noncal_heat}, and define the corresponding DN map \eqref{DN map} rigorously. We consider the Dirichlet problem for the following fractional parabolic equation:
\begin{align}\label{noncal_Diri}
\begin{cases}
\left( (\p_t-\Delta)^s +Q(t,x)\right)u(t,x)=F(t,x) & \hbox{in }\Omega_T,\\
u(t,x)=f(t,x) & \hbox{in }(\Omega_{e})_T, \\
u(t,x)=0 & \text{for }t\leq -T \text{ and }x\in \R^n,
\end{cases}
\end{align}
where we recall that $\Omega_T:= (-T,T)\times\Omega$ and $(\Omega_e)_T:= (-T,T)\times \Omega_e$ with $\Omega_e = \R^n \setminus \overline{\Omega}$.

Before studying the well-posedness of \eqref{noncal_Diri}, we emphasize the following observations:
\begin{itemize}
\item[(1)] We reiterate that the future data $u(t,x)|_{\{t\geq T\}}$ do \emph{not} affect the solutions of \eqref{noncal_Diri}. In particular, the following ``exterior" value problem, where we also prescribe data for $t\geq T$ (``in the future'') 
\begin{align}
\label{eq:main_eq}
\begin{cases}
\left( (\p_t-\Delta)^s +Q(t,x)\right)u(t,x)=F(t,x) & \hbox{in }\Omega_T,\\
u(t,x)=f(t,x) & \hbox{in }(\Omega_{e})_T \cup \{t\geq T\}, \\
u(t,x)=0 & \text{for }t\leq -T \text{ and }x\in \R^n,
\end{cases}
\end{align}
is \emph{not} an overdetermined problem. An alternative way of observing this will be provided by using the extensive characterization of the fractional operator (see Section \ref{Section 4}).

\item[(2)] We have shown that $u \chi _{[-T,T]}\in \HH^s(\R^{n+1})$ provided that $u\in \HH^s(\R^{n+1})$ in the previous section. Moreover, combined with the initial condition \eqref{noncal_Diri}, we know that $u \chi _{(-\infty,T]}\in \HH^s(\R^{n+1})$ provided that $u\in \HH^s(\R^{n+1})$
\end{itemize}

With these remarks in hand, we now proceed to the discussion of the well-posedness of \eqref{noncal_Diri}.

\subsection{Well-posedness}\label{well-posedness_statement}
We begin by setting up the weak formulation for the operator $\LL^s+Q$.
For $u,w\in \HH^s(\R^{n+1})$, we define 
the bilinear form $B_Q(\cdot,\cdot)$ by 
\begin{align}\label{bilinear}
    B_Q(u,w) &:=\left( \LL^{s/2}u, \LL^{s/2}_* w \right)_{\R^{n+1}}+\left( Qu, w\right)_{\Omega_T}.
\end{align} 
Moreover, by \eqref{integration by parts formula} and \eqref{bilinear} this can also be rephrased as 
\begin{align*} 
    B_Q(u,w) = \left \langle \LL^{s } u, w \right\rangle_{\HH^{-s}\times \HH^s}+\left( Qu, w \right)_{\Omega_T}.
\end{align*}
 
To simplify notation, given any $T\in (0,\infty)$, let us set 
\begin{align*}
u_T(t,x):=u(t,x)\chi_{[-T,T]}(t), \text{ for }t\in \R \text{ and }x\in \R^n.
\end{align*}
With this convention, we define the notion of a solution to \eqref{noncal_Diri}.

\begin{definition}[Weak solutions]
\label{defi:weak_sol}
	Let $\Omega$ be a bounded open set in $\R^n$,  $T\in (0,\infty)$ and $\Omega_T=(-T,T)\times \Omega\subset\R^{n+1}$. Given $F\in (\HH^s_{\overline{\Omega_T}})^*$ and $f\in \HH^s((\Omega_e)_T)$, for $u\in \HH^s(\R^{n+1})$, we say that $u\in \HH^s(\R^{n+1})$
	is a weak solution of \eqref{noncal_Diri} if $v:=u-f \in \mathcal{H}^s_{\overline{\Omega_T}}$, and 
		$$
		B_Q(u,w)=\langle F, w  \rangle_{(\HH^s_{\overline{\Omega_T}})^*\times \HH^s_{\overline{\Omega_T}}}, \quad \text{ for any }w\in \HH^s_{\overline{\Omega_T}},
		$$
	or equivalently, 
	$$
	B_Q(v,w) = \left\langle F-\left( ( \p_t -\Delta)^s +Q\right) f,w \right\rangle_{(\HH^s_{\overline{\Omega_T}})^*\times \HH^s_{\overline{\Omega_T}}},
	$$
    for any $w\in \HH^s_{\overline{\Omega_T}}$.
\end{definition}

We show the well-posedness for the fractional parabolic problem \eqref{noncal_Diri} in $\HH^s(\R^{n+1})$. By the possibility of choosing the future date arbitrarily  (c.f. (1) in the comment above), in this context uniqueness only holds in the sense that $u_T(t,x):= u(t,x) \chi_{[-T,T]}(t)$ and $u(t,x)\chi_{(-\infty,T]}(t,x)$ are uniquely determined by \eqref{eq:main_eq}.

\begin{theorem}[Well-posedness]\label{Thm well-posed}
	Let $\Omega$ be bounded and open set in $\R^n$. Suppose that $Q=Q(t,x)\in L^\infty(\Omega_T)$. Let 
	$
	    B_Q(u,w)  
	$
	be the bilinear form defined by \eqref{bilinear} for $u,w\in \HH^s(\R^{n+1})$.
	\begin{enumerate}
		\item There is a countable set $\Sigma=\{\lambda_j\}^\infty_{j=1}$ of real numbers $\lambda_1\leq \lambda_2\leq \cdots\rightarrow +\infty$, such that given $\lambda\in \R\setminus\Sigma$, for any $F\in (\HH^s_{\overline{\Omega_T}})^*$ and $f\in \HH^s((\Omega_e)_T)$, there exists a unique solution $u_T\in \HH^s(\R^{n+1})$ with $(u-f)_T\in \HH^s_{\overline{{\Omega_T}}}$ and 
		$$
		    B_Q(u_T,w)-\lambda(u_T,w)_{\Omega_T} = \left\langle F,w \right\rangle_{(\HH^s_{\overline{\Omega_T}})^*\times \HH^s_{\overline{\Omega_T}}},
		$$
		for any $w\in \HH^s_{\overline{{\Omega_T}}}$. 
		Moreover, $u$ satisfies the following stability estimate 
		\begin{align}\label{stability estimate of well-posedness}
			 \|u_T\|_{\HH^s(\R^{n+1})}\leq C_0 \LC \|F\|_{\HH^{-s}(\Omega_T)} + \|f\|_{\HH^s((\Omega_e)_T)}\RC,
		\end{align}
		for some constant $C_0>0$ independent of $u$, $F$ and $f$.
		\item The condition \eqref{eigenvalue condition} holds if and only if $0 \notin\Sigma$.
		
		\item If $Q\geq 0$ a.e. in $\Omega_T$, then $\Sigma\subset \R_+$ and the condition \eqref{eigenvalue condition} always holds.
	\end{enumerate}
\end{theorem}

\begin{proof}
Let $v:=(u-f)_T$ where $f\in \HH^s((\Omega_e)_T)$ denotes the exterior values of $u$. Then $v\in \HH^s_{\overline{{\Omega_T}}}$ and $v_T=v$. Considering the equation for $v$, it is sufficient to show that there exists a unique solution $v\in \HH^s_{\overline{{\Omega_T}}}$ such that 
	$$
	B_Q(v,w) +\mu ( v,  w)_{\Omega_T}=(\widetilde{F},w_T)
	$$
	for a suitable $\widetilde F\in (\HH^s_{\overline{{\Omega_T}}})^*$ and for any {$w\in \HH^s_{\overline{{\Omega_T}}}$}. 

We first prove the boundedness of the bilinear form $B_Q(v,w)$ in  $\HH^s(\R^{n+1})$ by showing that  
\begin{align}\label{boundedness of bilinear form}
|B_Q(v,w)|\leq C\|v\|_{\HH^s(\R^{n+1})}\|w\|_{\HH^s(\R^{n+1})}.
\end{align}
Indeed, by the definition of $B_Q(v,w)$, the Plancherel theorem and H\"older's inequality, we obtain that 
\begin{align*}
\left|(\widehat{\LL^{s/2}v}, \widehat{\LL_*^{s/2}w})_{\R^{n+1}}\right|
=& \left|\int_{\R^{n+1}}(i\rho+|\xi|^2)^{s/2}\widehat{v}(\rho,\xi)\overline{(-i\rho+|\xi|^2)^{s/2}\widehat{w}(\rho,\xi)}d\rho d\xi \right|  \\
\leq & \left( \int_{\R^{n+1}} (\rho^2+|\xi|^4)^{s/2}   |\widehat{v}(\rho,\xi)|^2d\rho d\xi\right)^{1/2} \\
&\times \left( \int_{\R^{n+1}} (\rho^2+|\xi|^4)^{s/2}   |\widehat{w}(\rho,\xi)|^2d\rho d\xi\right)^{1/2} \\
\leq & \|v\|_{\HH^s(\R^{n+1})}\|w\|_{\HH^s(\R^{n+1})},
\end{align*}
which proves \eqref{boundedness of bilinear form}.
Thus, 
\begin{align*}
\left|B_Q(v,w)+\mu(v,w)_{\Omega_T}\right|\leq C \|v\|_{\HH^s(\R^{n+1})}\|w\|_{\HH^s(\R^{n+1})},\text{ for }v,w\in \HH^s_{\overline{{\Omega_T}}}.
\end{align*}

It remains to prove the coercivity of $B_Q$ in the space $\HH^s_{\overline{{\Omega_T}}}$. From \eqref{integration by parts formula} and the Plancherel formula, one has

\begin{align*}
B_Q(v,v)+\mu(v,v)&= (\LL^{s/2} v, \LL^{s/2}_*v)_{\R^{n+1}} + (Qv|_\Omega, v|_\Omega)_{\Omega_T}+\mu(v,v)\\
& \geq  (\LL^{s/2} v, \LL^{s/2}_*v)_{\R^{n+1}}  ,
\end{align*}
where $\mu= \|-\min \{Q,0\}\|_{L^\infty(\Omega_T)}$. Further,
\begin{align*}
( \LL^{s/2}v ,  \LL_*^{s/2}v )_{\R^{n+1}}
&= \int_{\R^{n+1}}(i\rho+|\xi|^2)^{s/2}\widehat{v}(\rho,\xi)\overline{(-i\rho+|\xi|^2)^{s/2}\widehat{v}(\rho,\xi)}d\rho d\xi  \\
&=  \int_{\R^{n+1}} (\rho^2+|\xi|^4)^{s/2}(\cos(s\theta)+i\sin(s\theta))   |\widehat{v}(\rho,\xi)|^2d\rho d\xi \\
&= \int_{\R^{n+1}} (\rho ^2+|\xi|^4)^{s/2} \cos(s\theta) |\widehat{v}(\rho,\xi)|^2d\rho d\xi , \ \ 0<s<1,
\end{align*}
where $\tan\theta =\rho / {|\xi |^2}$ and where we utilized that $\sin(s\rho / {|\xi |^2})$ is an odd function in the last identity. By the definition of $\rho / {|\xi |^2}\in \R$ and $\theta =\tan^{-1}(\rho /{|\xi |^2})$, one has $-\pi/2<\theta<\pi/2$. It follows that $\cos\theta\geq 0$ . Thus, for any fixed $s\in(0,1)$, we know that $\cos(s\theta )\geq c_s:= \cos (\frac{s\pi}{2})>0$. This implies that  
\begin{align*}
B_Q(v,v)+\mu (v,v)_{\Omega_T} \geq c_s \int_{\R^{n+1}}(\rho^2+|\xi|^4)^{s/2}|\widehat{v}(\rho,\xi)|^2 d\rho d\xi.
\end{align*}

Next, we seek to prove that
\begin{align}\label{estimate 2}
\int_{\R^{n+1}}(\rho^2+|\xi|^4)^{s/2}|\widehat{v}(\rho,\xi)|^2 d\rho d\xi\geq C\|v\|_{L^2(\R^{n+1})}^{2},
\end{align} 
for some constant $C>0$. Note that the Fourier transform can be rewritten as 
$$\widehat{v}(\rho,\xi) =\F_x(\F_t v)(\rho,\xi).$$
It is easy to see that 
\begin{align*}
\int_\R\int_{\R^n} |i\rho+|\xi|^2|^{s} |\mathcal{F}_t\F_x v(\rho,\xi)|^2d\xi d\rho
\geq \int_\R  \LC \int_{\R^n}   |\xi|^{2s} |\mathcal{F}_x\F_t v(\rho,\xi)|^2d\xi \RC d\rho.
\end{align*}
By using the Plancherel's theorem and the Hardy-Littlewood-Sobolev inequality for the $x$-variable, we have
\begin{align*}
\int_\R  \LC \int_{\R^n}   |\xi|^{2s} |\widehat{v}(\rho,\xi)|^2d\xi \RC d\rho
&= \int_\R  \LC \int_{\R^n}   |\xi|^{2s} |\mathcal{F}_x (\mathcal{F}_t v)(\rho,\xi)|^2d\xi \RC d\rho\\
&=\int_\R  \LC \int_{\R^n}   |\F_x[(-\Delta_x)^{s/2}  (\mathcal F_t v)(\rho,\cdot) ]|^2d\xi \RC d\rho\\
&=\int_\R   \|(-\Delta_x)^{s/2} (\mathcal F_t v)(\rho,x)\|^2_{L^2(\R^{n})} d\rho \\
&\geq C\int_\R \|(\mathcal F_t v)(\rho,x)\|^2_{L^2(\R^{n})} d\rho\\
&= C\|v(t,x)\|_{L^2(\R^{n+1})}^2,
\end{align*}
where we have used that $v(t,x)$ is compactly supported in the $x$-variable.
Therefore,
\begin{align*}
\int_\R\int_{\R^n} |i\rho+|\xi|^2|^{s} |\mathcal{F}_t\F_x v(\rho,\xi)|^2d\xi d\rho \geq C \|v(t,x)\|_{L^2(\R^{n+1})}^2,
\end{align*}
which demonstrates \eqref{estimate 2}.
	
Hence, the bilinear form $B_Q(v,w)+\mu(v,w)_{\Omega_T}$ is bounded and coercive. Thus, by the Lax-Milgram theorem there is a unique solution $v=G_{\mu}\widetilde F\in \HH^s_{\overline{\Omega_T}}$ such that 
    $$
    B_Q(v,w)+\mu(v,w)_{\Omega_T}= (\widetilde F,w)
    $$
    for all $w\in \HH^s_{\overline{{\Omega_T}}}$. Moreover, 
	$$
	\|v\|_{\HH^s_{\overline{{\Omega_T}}}}\leq C_1  \|\widetilde F\|_{(\HH^s_{\overline{{\Omega_T}}})^*}
	$$
    for some constant $C_1$ independent of $\widetilde F$. 
    
    In particular, $G_{\mu}: (\HH^s_{\overline{{\Omega_T}}})^*\rightarrow \HH^s_{\overline{{\Omega_T}}}$ is bounded and by the compact Sobolev embedding, the operator  $G_{\mu}:L^2(\Omega_T)\rightarrow L^2(\Omega_T)$ is compact. Then the spectral theorem implies that the eigenvalues of $G_{\mu}$ are ${1\over \lambda_j+\mu}$ with $\lambda_j\rightarrow +\infty$ and $\Sigma\subset(-\|-(Q\wedge 0)\|_{L^\infty(\Omega_T)}, \infty)$. 
    This completes the proof of $(1)$.
 
 The claim of (2) is a direct consequence of (1) and the Fredholm alternative.

In order to deduce $(3)$, in the above proof of (1) we may choose $\mu=0$. Therefore, $\Sigma\subset(0, \infty)$, which implies that $0\notin\Sigma$. Thus \eqref{eigenvalue condition} holds by using condition $(2)$.
\end{proof}

\begin{remark} \label{Remark of well-posedness}
A number of remarks are in order:
	\begin{itemize}
		\item[(1)] Due to the form of the fractional parabolic operator $\LL^s$ (the operator couples the space and time variables), our well-posedness proof (which relies on the Lax-Milgram lemma) rather resembles a well-posedness proof for an \emph{elliptic} than for a \emph{parabolic} operator (for which one would typically use a Galerkin type approximation). 
		
		\item[(2)] We reiterate that in the stability estimate \eqref{stability estimate of well-posedness}, one can only hope to control the solution $u(t,x)$ for $\{t\leq T\}$, since the future data are independent of the fractional parabolic equation \eqref{noncal_Diri} and can hence still be chosen arbitrarily.
		
		\item[(3)] Considering the proof of Proposition \ref{well-posedness_statement}, we note that in principle it is not necessary to consider zero initial data (i.e. to prescribe $u=0$ for $t\leq -T$). It would also have been possible to prescribe initial data $u=g$ for $t\leq -T$ with $g\in \mathcal{H}^s((-\infty,-T)\times\R^n)$. In this case we would reduce the initial data to a compactly supported function, and thus would consider the function $\tilde{v}:= u-f- g\chi_{(-\infty,-T)}$ to prove a corresponding well-posedness result.

	\end{itemize}
\end{remark}

\begin{remark}
\label{rmk:adjoint}	
 Similarly, one can prove the following well-posedness for the adjoint parabolic equation. Let $Q=Q(t,x)\in L^\infty(\Omega_T)$ and $g=g(t,x) \in \HH^s((\Omega_e)_T)$. We seek to derive the existence and uniqueness of the solution $u\in \HH^s(\R^{n+1})$ of 
	\begin{align}\label{adjoint fractional parabolic equation}
	\begin{cases}
	\left( \LL^s_\ast + Q(t,x)\right)u(t,x)=0 & \text{ in }\Omega_T, \\
	u(t,x)=g(t,x) & \text{ in } (\Omega_e)_T, \\
	u(t,x)=0 & \text{ for }t\geq T \text{ and }x\in \R^n.
	\end{cases}
	\end{align}
	By considering the function $v:=(u-g)\chi_{[-T,\infty)}$, repeating the arguments of Theorem \ref{Thm well-posed} and relying on the same bilinear form $B_Q(\cdot,\cdot)$ of \eqref{adjoint fractional parabolic equation}, it is possible to derive the same properties as in Theorem \ref{Thm well-posed}. In particular (by the Fredholm alternative), the eigenvalue condition \eqref{eigenvalue condition} is equivalent to the following eigenvalue condition of \eqref{adjoint fractional parabolic equation}:
	\begin{align*}
	\begin{cases}
	\hbox{If }u\in\HH^s(\R^{n+1})\ \hbox{solves }(\LL^s_\ast+Q)u=0\ \hbox{in }\Omega_T \\
	\hbox{with $u|_{(\Omega_{e})_T}=0$} \text{ and }u=0 \text{ for }(t,x)\in [T, \infty)\times \R^n, \\
	\hbox{then }u\equiv 0 \text{ in }(-T,\infty) \times \R^n.
	\end{cases}
	\end{align*} 
\end{remark}

Heading towards the discussion of the inverse problem under consideration, we define the abstract  trace space for our exterior Dirichlet data by
\begin{align}\label{trace space}
\mathbb X:=\HH^s([-T,T]\times \R^n)/\HH^s_{\overline{{\Omega_T}}}.
\end{align}
Every function $f\in\HH^s([-T,T]\times \R^n)$ is a member of the set of class representative $[f]\in \mathbb X$. To simplify the notation, we use $f$ to denote $[f]$.

\subsection{The Dirichlet-to-Neumann map}
Relying on the well-posedness property of $\LL^s+Q$ whenever the eigenvalue condition \eqref{eigenvalue condition} holds, we define the corresponding DN map $\Lambda_Q$ for $\LL^s+Q$ by means of the bilinear form $B_Q$ defined by \eqref{bilinear}. Analogously, one can also define the adjoint DN map $\Lambda_Q^*$ by utilizing the following natural pairing property 
\begin{align} 
\left\langle \Lambda_{Q}f,g\right\rangle _{\mathbb X^*\times \mathbb X} =\left\langle f,\Lambda_{Q}^*g\right\rangle _{\mathbb X\times \mathbb X^*} ,\quad \text{ for }\quad f, g\in \mathbb X.\label{symmetric DN map}
\end{align}
Recall that the existence of the DN maps $\Lambda_Q$ and $\Lambda_Q^\ast$ is guaranteed by \eqref{eigenvalue condition} and Section \ref{well-posedness_statement}.

\begin{proposition}[DN map]
	\label{prop:DNmap} For $0<s<1$, let $\Omega$ be a bounded open set in $\mathbb{R}^{n}$ for $n\geq 1$ and $\Omega_T:=(-T,T)\times \Omega$. Let $Q\in L^{\infty}(\Omega_T)$ satisfy the eigenvalue condition \eqref{eigenvalue condition}.
	Let $\mathbb X$ be the abstract trace space given in \eqref{trace space}. Define
	\begin{equation}
	\left\langle \Lambda_{Q}f,g\right\rangle _{\mathbb X^*\times\mathbb X} :=B_{Q}(u_f,g),\quad f,g\in \mathbb X,\label{eq:equvalent integration by parts}
	\end{equation}
	where $u_f\in \HH^{s}(\mathbb{R}^{n+1})$ is
	the solution of \eqref{noncal_Diri} with the
	Dirichlet data $f$ in $(\Omega_e)_T$. Then 
	\[
	\Lambda_{Q}:\mathbb X\to \mathbb X^{*},
	\]
	is a bounded operator. Moreover, the adjoint DN map $\Lambda_Q^*$ can also be represented as 
	\begin{align}\label{adjoint and bilinear form}
    \langle f, \Lambda_Q^*g\rangle_{\mathbb X\times \mathbb X^*}=B_Q(f,u_g), 
	\end{align}
	where $u_g\in \HH^s(\R^{n+1})$ is the solution of the adjoint equation $(\LL^s_*+Q)u_g =0 $ in $\Omega_T$ with $u_g=g$ in $(\Omega_e)_T$.
\end{proposition}

\begin{proof}
For any $\phi,\ \psi\in \HH^s_{\overline{\Omega_T}}$, by Definition \ref{defi:weak_sol} and the definition of $\mathbb{X}$, 
$$
    B_Q(u_{f+\phi}, g+\psi)=B_Q(u_f, g).
$$
Thus, $\Lambda_{Q}$ is well-defined. Invoking this together with the $\mathcal{H}^s(\R^{n+1})$ bounds for the bilinear form $B_Q(\cdot, \cdot)$ yields 
\begin{align*}
\left|\left\langle \Lambda_{Q}f,g\right\rangle_{\mathbb X^*\times \mathbb X} \right|=\left|B_Q(u_f,g)\right| & \leq C\|f\|_{\mathbb X}\|g\|_{\mathbb X},
\end{align*}
for some constant $C>0$, whence $\Lambda_Q$ is bounded. 
Finally, due to the definition  \eqref{symmetric DN map} of $\Lambda_Q^*$ and the bilinear form $B_Q(\cdot, \cdot)$,  \eqref{adjoint and bilinear form} holds immediately.
\end{proof}

\begin{remark}
We emphasise that in contrast to the situation where the nonlocal operator is either $(-\Delta)^s+q$ or $(-\nabla \cdot (A\nabla ))^s+q$ for $q=q(x)$ (nonlocal elliptic operators) (see \cite{ghosh2017calderon,ghosh2016calder} for instance) our operator $\LL^s=(\p_t-\Delta)^s$ is not self-adjoint.
\end{remark}

\begin{remark}
\label{rmk:characterization}
In order to find an explicit distributional representation of the DN map $\Lambda_Q$ and thus justifying the expression in \eqref{DN map0}, we note that for any $g\in C_c^{\infty}((\Omega_e)_T)$ we have

\begin{align}
\langle\Lambda_{Q}f,g\rangle_{\mathbb X^{*}\times \mathbb X} & =B_{Q}(u_{f},g)\nonumber \\
& = \langle \LL^{s}u_{f}, g  \rangle_{\mathcal{H}^{-s}(\R^{n+1})\times \mathcal{H}^s(\R^{n+1})}+\int_{\Omega_T}(Qu_{f})\overline{g}dxdt\nonumber \\
& = \langle \LL^{s}u_{f}, g\rangle_{\mathbb X^{*}\times \mathbb X}.\label{t15}
\end{align}
By \eqref{eq:equvalent integration by parts} and \eqref{t15}, we can thus conclude that 
\begin{equation*}
\Lambda_{Q}f|_{(\Omega_e)_T}=\left.\mathcal{L}^{s}u_{f}\right|_{(\Omega_{e})_T},
\end{equation*}
where $u_f\in \HH^s(\R^{n+1})$ is a weak solution of $(\LL^s+Q)u_f=0$ in $\Omega_T$ with $u_f=f$ in $(\Omega_{e})_T$, $u_f=0$ for $\{t\leq {-T} \}$. 
Similarly, one can also derive that the adjoint DN map can expressed explicitly by $\Lambda_Q^* g|_{(\Omega_e)_T}= \LL^s_* u_g|_{(\Omega_e)_T}$, where $u_g\in \HH^s(\R^{n+1})$ is the unique solution of $(\LL^s_*+Q)u_g=0$ in $\Omega_T$ with $u_g=g$ in $(\Omega_{e})_T$ and $u_g=0$ for $\{t\geq T \}$.
\end{remark}

Last but not least, in concluding this section, we prove an associated Alessandrini type identity for the fractional parabolic equation. The Alessandrini type identity plays an essential role in proving the uniqueness and stability results, and we also refer readers to \cite{isakov1991completeness,isakov2006inverse} for this type identities for various PDEs.

\begin{lemma}[Integral identity]\label{Lem Integral identitiy}
Let $\Omega_T\subset \R^{n+1}$ be the bounded open set from above and let $Q_1,Q_2\in L^\infty(\Omega_T)$ satisfy the eigenvalue condition \eqref{eigenvalue condition}. Then, for any exterior Dirichlet data $f_1,f_2\in \mathbb X$ in $(\Omega_e)_T$, we have 
\begin{align*}
\left\langle (\Lambda_{Q_1}-\Lambda_{Q_2})f_1,f_2\right\rangle_{\mathbb X^*\times \mathbb X}=\left((Q_1-Q_2)u_1|_{\Omega_T},u_2|_{\Omega_T}\right)_{\Omega_T},
\end{align*}
where $u_1\in \HH^s(\R^{n+1})$ is a weak solution of $(\LL^s+Q_1)u_1=0$ in $\Omega_T$ with $u_1|_{(\Omega_e)_T}=f_1$ and $u_1=0$ for $\{t\leq -T \}$, and $u_2\in \HH^s(\R^{n+1})$ is a weak solution of $(\LL^s_{\ast}+Q_2)u_2=0$ in $\Omega_T$ with $u_2|_{(\Omega_e)_T}=f_2$ and $u_2 = 0 $ for $\{t\geq T\}$.
\end{lemma}

\begin{proof}
By the adjoint property \eqref{symmetric DN map} and \eqref{adjoint and bilinear form}, one has 
\begin{align*}
\left\langle (\Lambda_{Q_1}-\Lambda_{Q_2})f_1,f_2\right\rangle_{\mathbb X^* \times \mathbb X}&=\left\langle \Lambda_{Q_1}f_1,f_2\right\rangle_{\mathbb X^* \times \mathbb X}-\left\langle f_1,\Lambda_{Q_2}^*f_2\right\rangle_{\mathbb X \times \mathbb X^*} \\
&=B_{Q_1}(u_1,u_2)-B_{Q_2}(u_1,u_2)\\
&=\left((Q_1-Q_2)u_1|_{\Omega_T},u_2|_{\Omega_T}\right)_{\Omega_T}. \qedhere
\end{align*}
\end{proof}

%%%%%%%%%%%    Extension Problem      %%%%%%%%%%%%%%%%%%%%%%%%%%%%%%%%%%%%%
\section{The degenerate parabolic extension problem for $\LL^s$}\label{Section 4}

In this section, we recall that also for the fractional parabolic operator $(\p_t -\D)^s$ there is a parabolic \emph{Caffarelli-Silvestre extension}, which allows us to ``localize'' the problem at hand. In proving the weak unique continuation property and hence the desired Runge approximation result, we heavily exploit this.

In order to have an appropriate functional analytic set-up at our disposal, we introduce the following function space which is adapted to the Caffarelli-Silvestre extension:

\begin{definition}
	We define the function space $\W(\R\times \R^{n+1}_+)$ as follows:  
	\begin{align*}
	 &\quad \W(\R\times \R^{n+1}_+) \\
	 &	:=\left\{F \in L^2_{loc}(\R\times \R^{n+1}_+,x_{n+1}^{1-2s}):\ \p_{x_j} F, \ \p_{n+1} F\in L^2(\R\times \R^{n+1}_+,x_{n+1}^{1-2s}),\ j=1,\cdots,n\right\},
	\end{align*} 
	where $x'=(x_1,\cdots, x_n)\in \R^n$ and we use $\p_{n+1}=\partial_{x_{n+1}}$.  
		In particular, $\W(\R\times \R^{n+1}_+)$ is a Hilbert space endowed with the scalar product
	$$
	\langle F,G\rangle_{\W(\R\times \R^{n+1}_+)} = \int_{\R}\int_{\R^{n+1}_+} \LC FG+(\p_{n+1}F)(\p_{n+1}G)+\sum^{n}_{j=1}(\p_{x_i}F)(\p_{x_i}G)\RC x_{n+1}^{1-2s}dx'dx_{n+1}dt.
	$$
	If $F=F(t,X)\in \W(\R\times \R^{n+1}_+)$ with $X=(x',x_{n+1})$, we define its norm to be
	$$
	\|F\|_{\W(\R\times \R^{n+1}_+)}  = \langle F,F\rangle^{1/2}_{\W(\R\times \R^{n+1}_+)}=\LC\int_{\R} \|F(t,\cdot,\cdot)\|^2_{H^1\LC\R^{n+1}_+,x_{n+1}^{1-2s}\RC}dt\RC^{1/2}.
	$$   
Moreover, for any open set $\mathcal O\subset \R\times \R^{n+1}_+$, we define the space
$$
\W(\mathcal O)=\left\{F|_{\mathcal O}:\ F\in \W(\R\times \R^{n+1}_+)\right\}.
$$
\end{definition} 

With this definition in hand, we can formulate the existence of a parabolic Caffarelli-Silvestre extension operator.

\begin{proposition}
\label{prop:gen_Caff_Silv}
Let $s\in (0,1)$ and let $u \in \mathcal{H}^s(\R^{n+1} )$. Then there exists an extension operator
\begin{align*}
E_s : \mathcal{H}^s(\R \times \R^{n+1}_+) \rightarrow \mathcal{W}(\R \times \R^{n+1}_+), \ E_s u = \tilde{u}
\end{align*}
with the properties that $\tilde{u}$ is a weak solution to
\begin{align}
\label{eq:CS}
\begin{cases}
\left(x_{n+1}^{1-2s}\p_t - \nabla \cdot x_{n+1}^{1-2s}\nabla\right) \tilde{u} = 0& \mbox{ in } \R \times \R^{n+1}_+,\\
\tilde{u}  = u & \mbox{ on } \R \times \R^n \times \{0\},
\end{cases}
\end{align}
for which the following estimates hold:
\begin{align}
\label{eq:ext_bounds}
\begin{split}
\left\|\tilde{u}(\cdot, x_{n+1})-u(\cdot)\right\|_{\mathcal{H}^s( \R^{n+1})} &\rightarrow 0 \mbox{ as } x_{n+1} \rightarrow 0,\\
\lim\limits_{x_{n+1}\rightarrow 0}\left\| x_{n+1}^{1-2s} \p_{n+1} \tilde{u}(\cdot, x_{n+1})\right\|_{\mathcal{H}^{-s}( \R^{n+1})}
&\leq C \|u\|_{\mathcal{H}^s(\R \times \R^n)},\\
\left\|x_{n+1}^{1-2s}\p_{n+1}\tilde{u}(\cdot, x_{n+1})-d_s(\p_t - \D)^s u\right\|_{\mathcal{H}^{-s}(\R^{n+1})} &\rightarrow 0 \mbox{ as } x_{n+1}\rightarrow 0,\\
\left\|x_{n+1}^{\frac{1-2s}{2}} \nabla \tilde{u}\right\|_{L^2(\R \times \R^{n+1}_+)} 
&\leq C \|u\|_{\mathcal{H}^s(\R^{n+1})},\\
\left\|x_{n+1}^{\frac{1-2s}{2}} \tilde{u}\right\|_{L^2(\R \times \R^{n} \times (0,M))} 
&\leq C(M)\|u\|_{L^2( \R^{n+1})},
\end{split}
\end{align}
for some constant $C>0$ independent of $u, \tilde u$. Here $d_s=-\frac{2s\Gamma (-s)}{4^s \Gamma (s)}$, $M\in (0,\infty)$ is a finite number, and $C(M)>0$ is a constant depending $M$.
\end{proposition}

\begin{proof}
\emph{Step 1: Representation of the solution.}
We first assume that $u\in \mathcal{S}(\R^{n+1})$. Then Fourier transforming the equation \eqref{eq:CS} in time and in tangential directions in space leads to the ODE
\begin{align*}
\begin{cases}
i\rho \F \tilde{u} +|\xi|^2 \F \tilde{u} - (1-2s) x_{n+1}^{-1} \p_{n+1}(\F \tilde{u})- \p_{n+1}^2(\F \tilde{u})  = 0 & \mbox{ for } x_{n+1} \in (0,\infty),\\
 \F \tilde{u}  = \F u  &\mbox{ for } x_{n+1}=0. 
\end{cases}
\end{align*}
As in the case of the fractional Laplacian, this ODE can be transformed into a modified Bessel equation (see \cite[Section 4]{BG2017}, \cite[Appendix A]{GFRue19} and \cite[Section 4]{RS17a}). Searching for a function with decay at infinity
leads to
\begin{align*}
\F \tilde{u}( \rho,\xi, x_{n+1})
= C_s \F u(\rho,\xi ) (|\xi|^2+ i \rho)^{\frac{s}{2}} x_{n+1}^s K_s((|\xi|^2+ i \rho)^{\frac{1}{2}} x_{n+1}),
\end{align*}
where $K_s$ denotes the modified Bessel function of the second kind. We next prove that this representation and the asymptotics of the Bessel function imply the estimates of the Proposition for $u\in \mathcal{S}(\R^{n+1})$. In a final step, we extend the identities to $u \in \mathcal{H}^s(\R^{n+1})$ by density.\\

\emph{Step 2: Derivation of the estimates.}
The estimates follow from the asymptotics of the modified Bessel functions. Indeed, we note that 
\begin{align}
\label{eq:asympt}
\begin{split}
&\frac{d}{dt}(t^s K_s(t)) = c_s t^s K_{s-1}(t), \
K_s(t) \sim t^{-s} \mbox{ as } t \rightarrow 0,\
K_s(t) \sim \sqrt{\frac{\pi}{2t}} e^{-t} \mbox{ as } t \rightarrow \infty \mbox{ for } s \geq 0,\\
&K_s(t) = K_{-s}(t) \mbox{ for } s \leq 0, \\
\end{split}
\end{align}
With \eqref{eq:asympt} in hand, let us for instance prove the second and third estimates in \eqref{eq:ext_bounds} (the remaining ones are obtained similarly; we refer to \cite{GFRue19} for the analogues in the elliptic setting and also \cite{BG2017}). For the second bound in \eqref{eq:ext_bounds}, we note that by \eqref{eq:asympt} and $s\in (0,1)$, we have
\begin{align*}
& \quad x_{n+1}^{1-2s} \p_{n+1} \F \tilde{u}(\rho,\xi, x_{n+1} )\\
&= c_s x_{n+1}^{1-2s} (|\xi|^s + i \rho)^{\frac{s}{2}} x_{n+1}^s K_{s-1}((|\xi|^2 + i \rho)^{\frac{1}{2}}x_{n+1}) \F u(\rho,\xi ) (|\xi|^2 + i\rho)^{\frac{1}{2}}\\
& = c_s x_{n+1}^{1-s} (|\xi|^2 + i \rho)^{\frac{s+1}{2}} K_{1-s}((|\xi|^2 + i \rho)^{\frac{1}{2}} x_{n+1}) \F u(\rho,\xi ).
\end{align*}

Using the bounds for $K_{s-1}$ and denoting the homogeneous Sobolev spaces by $\dot{\mathcal{H}}^s(\R \times \R^n)$, we estimate as follows
\begin{align*}
&\quad \left\|x_{n+1}^{1-2s} \p_{n+1} \tilde{u}(\cdot, x_{n+1})\right\|_{\mathcal{H}^{-s}( \R \times \R^n)}
\leq \left\|x_{n+1}^{1-2s} \p_{n+1} \tilde{u}(\cdot, x_{n+1})\right\|_{\dot{\mathcal{H}}^{-s}( \R \times \R^n)}\\
&
\leq c_{s} \left\|x_{n+1}^{1-s} \left||\xi|^2 + i \rho\right|^{\frac{1}{2}} K_{1-s}((|\xi|^2 + i \rho)^{\frac{1}{2}}x_{n+1}) \F u \right\|_{L^2\left(\left\{(\rho,\xi)\in \R \times \R^n: \ ||\xi|^2 +i \rho|^{\frac{1}{2}}\geq \frac{\epsilon}{x_{n+1}}\right\}\right)}\\
& \quad + c_s \left\|x_{n+1}^{1-s} \left||\xi|^2 + i \rho\right|^{\frac{1}{2}} K_{1-s}((|\xi|^2 + i \rho)^{\frac{1}{2}}x_{n+1}) \F u\right\|_{L^2\left(\left\{ (\rho, \xi) \in \R \times \R^n: \ ||\xi|^2 + i \rho|^{\frac{1}{2}}< \frac{\epsilon}{x_{n+1}}\right\}\right)}\\
& \leq c_s \sup\limits_{|z|>\epsilon} \left|z^{1-s}K_{1-s}(z)\right| \left\|\left||\xi|^2 + i\rho\right|^{\frac{s}{2}}\F u\right\|_{L^2\left(\left\{(\rho,\xi): \ ||\xi|^2 + i \rho|^{\frac{1}{2}}> \frac{\epsilon}{x_{n+1}} \right\}\right)}    \\
& \quad + c_s \sup\limits_{|z|\leq \epsilon} \left|z^{1-s}K_{1-s}(z)\right| \|u\|_{\mathcal{H}^s( \R^{n+1})}. 
\end{align*}
Hence, as $x_{n+1}\rightarrow 0$,
\begin{align*}
 \left\|x_{n+1}^{1-2s} \p_{n+1} \tilde{u}(\cdot, x_{n+1})\right\|_{\mathcal{H}^{-s}( \R^{n+1})}
& \leq 2 c_s \sup\limits_{|z|\leq \epsilon} \left|z^{1-s}K_{1-s}(z)\right| \|u\|_{\mathcal{H}^s( \R^{n+1})},
\end{align*}
where we used that since $u \in \mathcal{H}^s( \R^{n+1})$, it holds
\begin{align*}
 \left\|\left||\xi|^2 + i\rho\right|^{\frac{s}{2}}\F u\right\|_{L^2\left(\left\{(\rho,\xi): \ ||\xi|^2 + i \rho|^{\frac{1}{2}}> \frac{\epsilon}{x_{n+1}} \right\}\right)} \rightarrow 0 \mbox{ as }  x_{n+1}\rightarrow 0. 
 \end{align*}
Now as $x_{n+1} \rightarrow 0$, $ \sup\limits_{|z|\leq\epsilon} \left|z^{1-s}K_{1-s}(z)\right|$ is bounded by \eqref{eq:asympt} concluding the proof of the second estimate in \eqref{eq:ext_bounds}.

Analogously, we obtain that for $\tilde{c}_s \neq 0$ chosen appropriately, we have
\begin{align*}
&\quad \left\|\tilde{c}_s x_{n+1}^{1-2s} \p_{n+1} \tilde{u}(\cdot, x_{n+1})-(\p_t -\D)^s u\right\|_{\mathcal{H}^{-s}( \R^{n+1})}\\
& = \left\|\left||\xi|^2 + i \rho\right|^{-\frac{s}{2}}(c_s \tilde{c}_s x_{n+1}^{1-s}(|\xi|^2 + i\rho)^{\frac{s+1}{2}}K_{1-s}((|\xi|^2 + i\rho)^{\frac{1}{2}}x_{n+1}) -(i\rho + |\xi|^2)^{s})\F u \right\|_{L^2(\R^{n+1})} \\
& \leq \left(c_s \tilde{c}_s \sup\limits_{|z|>\epsilon}|z^{1-s} K_{1-s}(z)|+1\right) \left\|\left||\xi|^2 + i\rho\right|^{\frac{s}{2}} \F u\right\|_{L^2\left(\left\{(\rho,\xi): \ ||\xi|^2 + i \rho|^{\frac{1}{2}}> \frac{\epsilon}{x_{n+1}} \right\}\right)}\\
& \quad +  \sup\limits_{|z|\leq \epsilon} \left|c_s\tilde{c}_s z^{1-s} K_{1-s}(z)-1\right| \left\|\left||\xi|^2 + i \rho\right|^{\frac{s}{2}} \F u\right\|_{L^2(\R^{n+1})}.
\end{align*}
Choosing $c_s\tilde{c}_s \neq 0$ in such a way that $\tilde{c}_s z^{1-s}K_{1-s}(z)\rightarrow 1$ as $z\rightarrow 0$ implies the claim by first letting $x_{n+1}\rightarrow 0$ and then $\epsilon \rightarrow 0$.

The arguments for the other estimates are similar.\\

\emph{Step 3: Extension to $u\in \mathcal{H}^s(\R^{n+1})$.}
For $u \in \mathcal{H}^s( \R^{n+1})$ the bulk estimates in \eqref{eq:ext_bounds} imply that for any sequence $u_k \in \mathcal{S}(\R^{n+1})$ with $u_k \rightarrow u$ in $\mathcal{H}^s(\R^{n+1})$ a limit $\tilde{u}$ of the functions $\tilde{u}_k:= E_s u_k$ exists in $\mathcal{W}(\R \times \R^{n+1}_+)$. Moreover, the first bound in \eqref{eq:ext_bounds} implies that $\tilde{u}(t,x',x_{n+1})\rightarrow u(t,x')$ in $\mathcal{H}^s( \R^{n+1})$. Using the weak form of the equation \eqref{eq:CS} one also obtains that $\tilde{u}$ solves this weakly. Finally, the second and third estimates in \eqref{eq:ext_bounds} yield that $\lim\limits_{x_{n+1}\rightarrow }x_{n+1}^{1-2s} \p_{n+1} \tilde{u}$ exists in $\mathcal{H}^{s}( \R^{n+1})$ and the equality
\begin{align*}
\lim\limits_{x_{n+1}}x_{n+1}^{1-2s} \p_{x_{n+1}} \tilde{u} = d_s\mathcal{L}^s u
\end{align*}
holds (as $\mathcal{H}^{-s} (\R^{n+1})$ functions) for some constant $d_s$ depending only on $s\in (0,1)$.
\end{proof}

We recall that weak solutions to (the local version of) the extension problem satisfy Caccioppoli estimates for parabolic equations (in weighted Sobolev spaces). We remark that by a weak solution we simply mean a function $\tilde{u}\in \mathcal{W}(\R \times \R^{n+1})$ such that the equation \eqref{eq:CS} holds tested against $H^1_0(\R \times \R^{n+1}_+)$ functions (in the case of the Dirichlet problem) and tested against $H^1_0(\R \times \overline{\R^{n+1}_+})$ functions in the Neumann case (note that the resulting boundary terms in the Neumann case are well-defined as $H^1_0(\R \times \overline{\R^{n+1}_+}) \rightarrow \mathcal{H}^s(\R^{n+1})$ by the trace estimate). 
In the sequel (in particular in our Carleman estimates), the Caccioppoli estimates will allow us to control gradient contributions in terms of $L^2$ terms. 

Recall that we denote $x'=(x_1,\cdots,x_n)\in\R^n$. Let us introduce the following notation:
Given $r\in (0,\infty)$, $x_0\in \R^n$, we consider 
\begin{align*}
\begin{split}
& B_r^+(x_0,0):=\left\{X=(x',x_{n+1})\in \R^{n+1}: \ |(x',x_{n+1})-(x_0,0)|<r \right\}\cap \{x_{n+1}>0\}\subset \R^{n+1},  \\
& B_r'(x_0):= \left\{x'\in \R^n: \ |x'-x_0|<r \right\} \subset \R^n.
\end{split}
\end{align*}
In particular, when $x_0=0$, we simply denote $B_r^+ :=B_r^+(0,0)$ and $B'_r:=B'_r(0)$.
 
\begin{lemma}[Caccioppoli inequality]
\label{lem:Cacc}
Let $\tilde{u} \in \mathcal{W}((0,1)\times B_1^+)$ be a weak solution to 
\begin{align*}
\begin{cases}
x_{n+1}^{1-2s}\p_t \tilde{u} - \nabla \cdot (x_{n+1}^{1-2s} \nabla \tilde{u} ) = 0 & \mbox{ in } (0,1)\times B_1^+ ,\\
\tilde{u} = 0 & \mbox{ on } (0,1)\times B_1',
\end{cases}
\end{align*}
or to 
\begin{align*}
\begin{cases}
x_{n+1}^{1-2s}\p_t \tilde{u} - \nabla \cdot x_{n+1}^{1-2s} \nabla \tilde{u}  = 0 & \mbox{ in } (0,1)\times B_1^+  ,\\
\lim\limits_{x_{n+1}\rightarrow 0} x_{n+1}^{1-2s} \p_{n+1} \tilde{u}  = 0 & \mbox{ on } (0,1)\times B_1'.
\end{cases}
\end{align*}
Assume that $\eta \in C^{\infty}((0,1)\times B_1^+)$ with $\supp(\eta) \subset (0,1)\times \overline{B}_1^+$.
Then, 
\begin{align*}
&\quad \sup\limits_{t\in (0,1)}\left\|x_{n+1}^{\frac{1-2s}{2}}\eta \tilde{u}\right\|_{L^2(B_1^+)}^2
+ \left\|x_{n+1}^{\frac{1-2s}{2}}\eta |\nabla \tilde{u}|\right\|_{L^2((0,1)\times B_1^+ )}^2 \\
&\leq   C \left\|x_{n+1}^{\frac{1-2s}{2}}(|\eta \p_t \eta| + |\nabla \eta|^2)^{\frac{1}{2}} \tilde{u}\right\|_{L^2((0,1)\times B_1^+)}^2,
\end{align*}
for some constant $C>0$ independent of $\tilde u$.
\end{lemma}

\begin{proof}
We differentiate with respect to $t$ and use the equation for $\tilde{u}$:

\begin{align*}
\frac{d}{dt} \left\|x_{n+1}^{\frac{1-2s}{2}} \eta \tilde{u}\right\|_{L^2(B_1^+)}^2
&= 2\int\limits_{B_1^+} x_{n+1}^{1-2s} \left(\tilde{u} \eta^2 \p_t \tilde{u} + \tilde{u}^2 \eta \p_t \eta \right)dx\\
&=-2 \int\limits_{B_1^+} x_{n+1}^{1-2s}\nabla(\eta^2 \tilde{u})\cdot \nabla \tilde{u} dx  + 2 \int\limits_{B_1^+} x_{n+1}^{1-2s} \tilde{u}^2 \eta \p_t \eta dx\\
& =- 4 \int\limits_{B_1^+} x_{n+1}^{1-2s}  \tilde{u} \eta \nabla \eta \cdot \nabla \tilde{u} dx -2\int\limits_{B_1^+} x_{n+1}^{1-2s} \eta^2 |\nabla \tilde{u}|^2 dx + 2 \int\limits_{B_1^+}x_{n+1}^{1-2s} \tilde{u}^2 \eta \p_t \eta dx\\
& \leq -\int\limits_{B_1^+} x_{n+1}^{1-2s} \eta^2 |\nabla \tilde{u}|^2 dx + 4 \int\limits_{B_1^+} x_{n+1}^{1-2s} \tilde{u}^2 |\nabla \eta |^2  dx+ 2 \int\limits_{B_1^+}x_{n+1}^{1-2s} \tilde{u}^2 \eta \p_t \eta dx,
\end{align*} 
where we used the vanishing trace of $\tilde{u}$ (or of $\lim\limits_{x_{n+1}\rightarrow 0} x_{n+1}^{1-2s}\p_{x_{n+1}} \tilde{u}$) on $B_1'$ and applied Young's inequality.
Rearranging this, integrating in $t$ and using the support condition for $\eta$ then proves the claim.
\end{proof}

	Note that the Caccioppoli inequality also holds for the backward degenerate heat equation $x_{n+1}^{1-2s}\p_t \tilde v + \nabla \cdot (x_{n+1}^{1-2s}\nabla \tilde v)=0$ in $(0,1)\times B_1^+(0)$ which follows immediately from changing $t\mapsto -t$.

Next, we state the Schauder type estimates from \cite{BG2017} (see also \cite[Appendix A]{KRS16} for weighted Schauder type estimates), which can also be viewed as a consequence of the pseudolocality of our operator:

\begin{lemma}[Theorem 5.1 in \cite{BG2017}]
\label{lem:reg_up_to_boundary}
Let $V\in L^{\infty}((0,1)\times B_1')$ and $v \in \mathcal{W}((0,1)\times B_1^+)$ be a weak solution to
\begin{align*}
\begin{cases}
x_{n+1}^{1-2s}\p_t v - \nabla \cdot (x_{n+1}^{1-2s}\nabla v)  = 0& \mbox{ in } (0,1)\times B_1^+,\\
\lim\limits_{x_{n+1}\rightarrow 0} x_{n+1}^{1-2s} \p_{n+1} v  = Vv & \mbox{ in } (0,1) \times B_1'.
\end{cases}
\end{align*}
Then, for some $\alpha\in (0,1)$ and $\delta\in (0,1)$, we have that $\p_{x'} v,\p_t v, x_{n+1}^{1-2s}\p_{n+1}v  \in C^{\alpha,\alpha/2}((\delta,1-\delta)\times B_{\frac{1}{2}}')$, where $C^{\alpha, \alpha/2}$ denotes the parabolic H\"older space with exponent $\alpha \in (0,1)$. 

Moreover, if $V \in C^{k}((0,1)\times B_1')$, we have that $\p^{\alpha}_{t,x'} v$ and $x_{n+1}^{1-2s}\p_{t,x'}^{\alpha} \p_{x_{n+1}} v \in C^{\alpha,\alpha/2}((\delta,1-\delta)\times B_{\frac{1}{2}}')$, where $\p^{\alpha}_{t,x'}= \p_t^{\alpha_0}\p_{x_1}^{\alpha_1}\dots \p_{x_{n}}^{\alpha_n}$ for $\alpha \in \N^{1+n}$ with $2\alpha_0 + \alpha_1 + \dots + \alpha_n \leq k$.
\end{lemma}

\begin{proof}
The first statement is a direct result of the Schauder estimates in \cite{BG2017}. The higher regularity result in time and the tangential directions in space follows from considering difference quotients in time and space which is possible by the translation invariance in these directions.
\end{proof}

\begin{remark}
Note that the extension property also holds for the adjoint fractional parabolic operator $\LL^s_\ast$. As the arguments are analogous to the ones presented above, we do not discuss the details of this.
\end{remark}

\section{Unique continuation property}\label{Section 5}

In this section, we will show the global weak unique continuation property for the fractional parabolic operator $\LL^s$, which is stated in Theorem \ref{Thm UCP}. 
In order to prove the desired result, we transfer the unique continuation statement for the operator $\LL^s$ into a unique continuation statement for the extension operator from \eqref{eq:CS}. The problem hence turns into a (global) weak boundary unique continuation result for this operator.

Next, we introduce the notion of vanishing of infinite order that we are going to use in the sequel:

\begin{definition}\label{Def of vanishing to infinite}
We say that a function $\tilde{u}\in \mathcal{W}(\R\times \R^{n+1}_+)\cap C^0((t_0-r^2,t_0+r^2)\times \overline{B_r^+(x_0,0)})$, where $r>0$ is a small radius, \emph{strongly vanishes to infinite order at a point $(t_0,x_0,0)\in  \R\times \R^{n}\times \{x_{n+1}=0\}$} provided that 
$$
\lim_{x_{n+1} \to 0}x_{n+1}^{-m}\tilde{u}(t_0 ,x_0, x_{n+1})=0 \text{ for any }m>0.
$$
\end{definition}

Notice that when the function $\tilde{u}$ is locally  $C^{\infty}$-smooth (up to the bounday $\{x_{n+1}=0\}$), then the above definition is equivalent to the classical definition for a function $\tilde{u}$ which vanishes of infinite order at a given point, that is, $\tilde{u}$ and all its derivatives vanish at that point.

In the sequel, we reduce the statement of Theorem \ref{Thm UCP} to the global (boundary) weak unique continuation property for the extended parabolic problem and then provide an independent proof of this statement. Here we crucially rely on a Carleman estimate (see Proposition \ref{prop:Carl_eucl}). Before however turning to this, we first show that the Caffarelli-Silvestre extension of $u$ vanishes of infinite order at any boundary point $(0,r)\times B_{r}'\subset (0,1) \times \mathcal{U} \times \{0\} $, where $\mathcal{U}$ is an open subset of $\R^n$:

\begin{lemma}
\label{lem:inf_van_order}
Let $s\in (0,1)$ and $u \in \mathcal{H}^s( \R^{n+1})$. Assume that $u\equiv 0$ and $(\p_t -\D)^s u \equiv 0$ in $(0,1) \times \mathcal{U} $. Let $\tilde{u}$ denote the Caffarelli-Silvestre extension of $u$. Then we have that for any $m\in \N$
\begin{align*}
\lim\limits_{x_{n+1}\rightarrow 0} x_{n+1}^{-m} \tilde{u}(t,x',x_{n+1}) = 0 \mbox{ for } (t,x')\in (\delta,1-\delta)\times \mathcal{U}',
\end{align*}
for  some $\delta\in (0,1)$, where the open set $\mathcal{U}'$  is strictly contained in $\mathcal{U}\subset \R^n$.
\end{lemma}

The argument for this relies on the Schauder estimates from Lemma \ref{lem:reg_up_to_boundary} and a bootstrap argument.

\begin{proof}[Proof of Lemma \ref{lem:inf_van_order}]
Relying on the regularity estimates from Lemma \ref{lem:reg_up_to_boundary} and invoking a bootstrap argument as in \cite{ruland2015unique}, it is possible to prove that $\tilde{u}$ vanishes of infinite order at $(0,1) \times \mathcal{U} \times \{0\}$. We note that using difference quotient arguments, it is always possible to boostrap the tangential and temporal regularity of $\tilde{u}$ (see the second part of Lemma \ref{lem:reg_up_to_boundary}); in the sequel, we will make extensive use of this.\\
We discuss the details of this in the sequel. 
\medskip

\emph{Step 1: Initial regularity.}
First, by the fundamental theorem of calculus and by Lemma \ref{lem:reg_up_to_boundary}, we have that for some constant $C_{s,n}>0$
\begin{align*}
|\tilde{u}(t,x',x_{n+1})| 
&= \left| \int\limits_{0}^1 \p_{n+1}\tilde{u}(t,x', r x_{n+1}) dr \right| |x_{n+1}|\\
&\leq \sup\limits_{r\in (0,1)}\left(| r x_{n+1}|^{1-2s}\left|  \p_{n+1} \tilde{u}(t,x', r x_{n+1}) \right|\right) |x_{n+1}|^{2s}\int\limits_{0}^1 r^{2s-1}dr\\
&\leq C_{s,n}|x_{n+1}|^{2s}.
\end{align*}
As an immediate consequence of the translation invariance in the tangential $x'$ and the $t$ directions, we also directly obtain that similar estimates hold for the tangential and temporal derivatives in a slightly smaller space-time domain, i.e. we have that
\begin{align}
\label{eq:deriv_tang}
|x_{n+1}^{1-2s}\p_t \tilde{u}(t,x',x_{n+1})| + |x_{n+1}^{1-2s}\D' \tilde{u}(t,x',x_{n+1})| \leq C |x_{n+1}|,
\end{align}
for $(t,x')\in (\delta_1,1-\delta_1)\times \mathcal{U}_1'$, where $\delta_1 >0$ and $\mathcal{U}'_1 \subset \mathcal{U}'$.

Next, by the equation in the bulk and the tangential and temporal regularity of the solutions, we have for $x_{n+1}>0$
\begin{align*}
\p_{n+1} x_{n+1}^{1-2s}\p_{n+1} \tilde{u}
= -  x_{n+1}^{1-2s} \D' \tilde{u} + x_{n+1}^{1-2s} \p_{t} \tilde{u},
\end{align*}
whence, in combination with \eqref{eq:deriv_tang}, we infer
\begin{align}
\label{eq:double_normal}
|\p_{n+1} x_{n+1}^{1-2s} \p_{n+1} \tilde{u}(t,x',x_{n+1})|
\leq C |x_{n+1}|.
\end{align}
In particular, we may pass to the limit $x_{n+1}\rightarrow 0$ which implies that $\lim\limits_{x_{n+1}\rightarrow 0} \p_{n+1} x_{n+1}^{1-2s} \p_{n+1} \tilde{u}$ exists and $\lim\limits_{x_{n+1}\rightarrow 0} \p_{n+1} x_{n+1}^{1-2s} \p_{n+1} \tilde{u}=0$ (see also \cite[Appendix A]{KRS16} for weighted regularity estimates up to the boundary). 
Hence, combining \eqref{eq:double_normal} with a similar fundamental theorem argument as above and exploiting the regularity of $x_{n+1}^{1-2s}\p_{n+1}\tilde{u}$ as well as our remark on the tangential and temporal boostrap arguments yields that for $(t,x') \in (\delta_2,1-\delta_2) \times \mathcal{U}_2'$ (with $\delta_2 >\delta_1$ and $\mathcal{U}_2' \subset \mathcal{U}'_1$ to be specified) there exists a constant $C=C(n,s,\delta_1,\delta_2)>0$ such that 
\begin{itemize}
\item[(i)] $|\p_{n+1}\tilde{u}(t,x',x_{n+1})| \leq C |x_{n+1}|^{2s+1}$.
\item[(ii)] for $\tilde{v}(t,x',x_{n+1}):= x_{n+1}^{1-2s}\p_{n+1} \tilde{u}(t,x',x_{n+1})$ we have $|\tilde{v}(t,x',x_{n+1})| \leq C |x_{n+1}|^2$,
\item[(iii)] $|\p_{n+1}\tilde{v}(t,x',x_{n+1})| \leq C |x_{n+1}|$,
\item[(iv)] $|\tilde{u}(t,x',x_{n+1})| \leq C|x_{n+1}|^{2+2s}$,
\item[(v)] $|x_{n+1}^{-2s} \D' \tilde{u}(t,x',x_{n+1})| + |x_{n+1}^{-2s} \p_t \tilde{u}(t,x',x_{n+1})| \leq C|x_{n+1}|^2$.
\end{itemize}

\medskip
\emph{Step 2: Upgrade of the decay estimates through upgraded regularity estimates.}
With the estimates from Step 1 in hand, we seek to upgrade the regularity estimates by reducing the problem to a forced heat equation. Indeed, we note that for $x_{n+1}>0$ (where the equation is strictly parabolic and hence $\tilde{u}(t,x
',x_{n+1})$ is smooth), a differentiation with respect to the $x_{n+1}$-direction leads to the bulk equation
\begin{align*}
\begin{split}
\D(x_{n+1}^{1-2s} \p_{n+1}\tilde{u}) -  \p_t(x_{n+1}^{1-2s} \p_{n+1}\tilde{u}) & = -(1-2s)x_{n+1}^{-2s} \D' \tilde{u} + (1-2s)x_{n+1}^{-2s}\p_t \tilde{u}.
\end{split}
\end{align*}
We note that by the estimates from Step 1 the contributions on the right hand side are Hölder continuous and vanish as $x_{n+1}\rightarrow 0$. Furthermore, also by Step 1 (see \eqref{eq:double_normal}), we have that 
\begin{align*}
\lim\limits_{x_{n+1}\rightarrow 0} \p_{n+1} x_{n+1}^{1-2s} \p_{n+1} \tilde{u}=0.
\end{align*}
As a consequence, the function $\tilde{v}(t,x',x_{n+1}):= x_{n+1}^{1-2s} \p_{n+1} \tilde{u}(t,x',x_{n+1})$ is a (weak) solution to 
\begin{align}
\label{eq:equation_heat}
\begin{cases}
- \p_t \tilde{v} + \D \tilde{v}  = f_1& \mbox{ in }(\delta_2,1-\delta_2)\times \mathcal{U}'_2 \times (0,1/2),\\
\lim\limits_{x_{n+1}\rightarrow 0} \p_{n+1} \tilde{v}  = 0 & \mbox{ on } (\delta_2,1-\delta_2)\times \mathcal{U}'_2 \times\{0\},
\end{cases}
\end{align}
where
\begin{align*}
f_1= -(1-2s) x_{n+1}^{-2s} \D' \tilde{u} + (1-2s) x_{n+1}^{-2s} \p_t \tilde{u} .
\end{align*}
Using property (v) from Step 1 as well as Schauder theory for the heat equation, we obtain that $\tilde{v}\in C^2((\delta_3, 1-\delta_3) \times \mathcal{U}'_3 \times (0,1/2-\delta_3))$ (where $\delta_3 \geq \delta_2$ and $\tilde{U}'_3 \subset \tilde{U}_2'$ is to be determined).
 By virtue of the equation \eqref{eq:equation_heat} in combination with property (v) from Step 1, we further obtain the pointwise bound
\begin{align*}
|\p_{n+1}^2 \tilde{v}(t,x',x_{n+1})| \leq |f_1(t,x', x_{n+1})| 
\leq C_{s,n,\delta_2}|x_{n+1}|^2.
\end{align*}
Bootstrapping by means of the fundamental theorem and by recalling that analogous estimates can be obtained for the tangential spatial and temporal derivatives, we obtain that in $(\delta_3,1-\delta_3) \times \mathcal{U}'_3 \times ({0},1/2-\delta_3)$ and with $C=C(n,s,\delta_1,\delta_2,\delta_3)>0$
\begin{itemize}
\item[(a)] $|\p_{n+1}\tilde{v}(t,x',x_{n+1})| \leq C |x_{n+1}|^3$,
\item[(b)] $|\tilde{u}(t,x',x_{n+1})| \leq C |x_{n+1}|^{4+2s}$,
\item[(c)] $|\p_t \tilde{u}(t,x',x_{n+1})| + |\D' \tilde{u}(t,x',x_{n+1})| \leq C |x_{n+1}|\leq C|x_{n+1}|^{4+2s}$. 
\end{itemize}

Exploiting this, we can again differentiate the equation in the normal direction and bootstrap the argument correspondingly. Iterating this procedure, and choosing $\delta_{\ell} \rightarrow \delta$, $\mathcal{U}'_{\ell} \rightarrow \mathcal{U}'$ as $\ell \rightarrow \infty$, it ultimately implies the estimate
\begin{align*}
|\tilde{u}(t,x',x_{n+1})| \leq C_{m,s,n} |x_{n+1}|^m,
\end{align*}
for all $m \in \N$ and $(t,x',x_{n+1}) \in (\delta,1-\delta)\times \mathcal{U}' \times \{0\}$, and for some constant $C_{m,s,n}>0$. This however yields the desired infinite order of vanishing of $\tilde{u}$ in $(\delta,1-\delta)\times \mathcal{U}'  \times \{0\}$.
\end{proof}

With the vanishing of infinite order in hand, we next seek to prove that $\tilde{u} = 0$ in the upper half plane. To this end, we rely on a Carleman estimate which we deduce in the next section. Exploiting this, we will be able to exclude non-trivial behaviour of $\tilde{u}$ as $x_{n+1}\rightarrow 0$ and thus prove the desired (weak) boundary unique continuation result.

\subsection{A Carleman estimate for a fractional heat operator}
In this section we seek to deduce a Carleman estimate for the operator 
\begin{align}
\label{eq:heat_back}
x_{n+1}^{1-2s} \p_t + \nabla \cdot x_{n+1}^{1-2s} \nabla
\end{align}
with vanishing weighted Neumann data.
For convenience, we have here reversed the time direction. The proof of the Carleman estimate proceeds in two steps: First, we introduce suitable parabolic conformal coordinates. Then we carry out the conjugation argument yielding the desired Carleman estimates.

\subsubsection{Parabolic conformal coordinates}

In order to simplify the derivation of the estimate and to clarify the choice of the Carleman weight, we introduce parabolic conformal coordinates (see also \cite{KT09}):
\begin{align*}
t = e^{- 4 \ell}, \ x = 2 e^{- 2 \ell } y.
\end{align*}
A short computation then yields that 
\begin{align*}
\begin{pmatrix}
\frac{\p}{\p t} \\ \frac{\p}{\p x}
\end{pmatrix}
= \begin{pmatrix}
- \frac{1}{4} e^{4 \ell} & - \frac{1}{2} e^{4 \ell} y \\
0 & \frac{1}{2} e^{2 \ell}
\end{pmatrix}
\begin{pmatrix}	
\frac{\p}{\p \ell} \\
\frac{\p}{\p y}
\end{pmatrix}.
\end{align*}
Hence, the operator \eqref{eq:heat_back} transforms into 
\begin{align*}
\frac{(2 e^{- 2\ell})^{1-2s} e^{4 \ell}}{4}\left[ y_{n+1}^{1-2s} \left(
- \p_{\ell}- 2 y \cdot \nabla_{y} \right)
+ \nabla_{y} \cdot y_{n+1}^{1-2s} \nabla_y
 \right].
\end{align*}
Multiplying this with $4 e^{- 4 \ell} (2 e^{-2 \ell})^{2s-1}$ therefore leads to the operator
\begin{align}\label{UCP cal 1}
y_{n+1}^{1-2s} \left(
- \p_{\ell}- 2 y \cdot \nabla_{y} \right)
+ \nabla_{y} \cdot y_{n+1}^{1-2s} \nabla_{y} . 
\end{align}
Conjugating \eqref{UCP cal 1} by $e^{- \frac{|y|^2}{2}}$ then further results in 
\begin{align}\label{UCP cal 2}
&\notag e^{- \frac{|y|^2}{2}}
\left[ y_{n+1}^{1-2s} \left(
- \p_{\ell}- 2 y \cdot \nabla_{y} \right)
+ \nabla_{y} \cdot y_{n+1}^{1-2s} \nabla_{y} \right] e^{\frac{|y|^2}{2}}\\
= & y_{n+1}^{1-2s}\left( - \p_{\ell} - |y|^2 \right) + \nabla_y \cdot y_{n+1}^{1-2s} \nabla_y + (n+2-2s) y_{n+1}^{1-2s}.
\end{align}

In order to eliminate the zeroth order term of \eqref{UCP cal 2}, we conjugate \eqref{UCP cal 2} with $e^{-(n+2-2s)\ell}$, which gives
\begin{align}\label{UCP cal 3}
e^{-(n+2-2s)\ell}\left( y_{n+1}^{1-2s}\left( - \p_{\ell} -  |y|^2 \right) + \nabla_y \cdot y_{n+1}^{1-2s} \nabla_y  \right) e^{(n+2-2s) \ell }.
\end{align}
Finally, we multiply the operator \eqref{UCP cal 3} from the left and right by $y_{n+1}^{\frac{2s-1}{2}}$ and obtain
\begin{align*}
L:= - \p_{\ell} + y_{n+1}^{\frac{2s-1}{2}} \nabla_y \cdot y_{n+1}^{1-2s} \nabla_y y_{n+1}^{\frac{2s-1}{2}} - |y|^2 =: -\p_{\ell} - H_s,
\end{align*}
where in analogy to the case $s=\frac{1}{2}$, we refer to $H_s$ as the fractional Hermite operator.

We summarize this discussion in the following lemma:

\begin{lemma}
\label{lem:conformal_polar_coord}
Let $f: \R \times \R^{n+1}_+ \rightarrow \R$ and consider a function $u:\R \times \R^{n+1}_+  \rightarrow \R$. Then $u(t,x)$ is a solution to 
\begin{align*}
\left(x_{n+1}^{1-2s}\p_t + \nabla \cdot x_{n+1}^{1-2s} \nabla \right) u(t,x) = f(t,x) \mbox{ in } \R \times \R^{n+1}_+,
\end{align*} 
if and only if the function $w( y, \ell):= y_{n+1}^{\frac{1-2s}{2}} e^{-(n+2-2s)\ell}e^{-{|y|^2\over 2}}u(e^{- 4\ell},2 e^{- 2\ell} y)$ is a solution to
\begin{align*}
(\p_{\ell} + H_{s}) w(\ell,y) & = g(\ell,y) \mbox{ in } \R \times \R^{n+1}_+,
\end{align*}
where $g(\ell, y) = 4 e^{-4 \ell} (2 e^{-2\ell})^{2s-1}e^{- \frac{|y|^2}{2}} e^{-(n+2-2s)\ell} y_{n+1}^{\frac{2s-1}{2}} f(e^{-4 \ell}, 2 e^{-2\ell}y)$. 
\end{lemma}

With this in hand, we deduce a Carleman estimate in conformal polar coordinates:

\begin{proposition}
\label{prop:Carl_conf}
Let $h: \R \rightarrow \R, \ \ell \mapsto h(\ell)$ be a convex, asymptotically linearly growing function. Assume further that $w \in L^2(\R\times\R^{n+1}_+)\cap {C^{\infty}_{loc}(\R \times \R^{n+1}_+)}$ is decaying superlinearly as $|\ell | \rightarrow  \infty$ and $|y| \rightarrow \infty$ and satisfies
\begin{align*}
\begin{cases}
\left(\p_{\ell}-y_{n+1}^{\frac{2s-1}{2}} \nabla_y \cdot y_{n+1}^{1-2s} \nabla_y y_{n+1}^{\frac{2s-1}{2}} + |y|^2\right) w = {f}(\ell,y) & \mbox{ in } \R \times \R^{n+1}_+,\\
\lim\limits_{y_{n+1}\rightarrow 0} y_{n+1}^{1-2s} \p_{n+1} (y_{n+1}^{\frac{2s-1}{2}} w) = 0 & \mbox{ on } \R \times \R^{n} \times \{0\},
\end{cases}
\end{align*}
where $f(\ell,y) \in L^2(\R \times \R^{n+1}_+)$ has superlinear decay as $|\ell| \rightarrow \infty$ and $|y| \rightarrow \infty$.
Then, there exists $C>0$ such that for all $\tau \geq \tau_0>0$
\begin{align*}
\tau \|e^{\tau h} (h'')^{\frac{1}{2}} w\|_{L^2(\R \times \R^{n+1}_+ )}^2
\leq C\|e^{\tau h}(\p_{\ell}+H_s) w\|_{L^2(\R \times \R^{n+1}_+)}^2.
\end{align*}
\end{proposition}

\begin{proof}
This follows from a conjugation argument. Indeed, we have
\begin{align*}
L_h:= e^{\tau h(\ell)}(\p_{\ell} + H_s ) e^{-\tau h(\ell)}
= \p_{\ell} + H_s - \tau h'.
\end{align*}
Then, up to boundary terms this entails that the symmetric and antisymmetric parts of this operator are given by
\begin{align*}
S= H_s - \tau h', \qquad A = \p_{\ell}.
\end{align*}
As a consequence, we obtain that 
\begin{align}
\label{eq:expand_a}
\begin{split}
\|L_h v\|_{L^2(\R \times \R^{n+1}_+)}^2
& = \|S v\|_{L^2(\R \times \R^{n+1}_+)}^2 + \|A v\|_{L^2(\R \times \R^{n+1}_+)}^2 + 2(S v, Av)\\
 & = \|S v\|_{L^2(\R \times \R^{n+1}_+)}^2 + \|A v\|_{L^2(\R \times \R^{n+1}_+)}^2 + ([S,A] v, v) + \text{(BT)},
\end{split}
\end{align}
where $[S,A]:=SA-AS$ denotes the commutator and (BT) are boundary correction terms. These are obtained as boundary terms in the integration by parts estimates which lead from $2(S v, Av)$ to $([S,A] v, v)$. Here only the spatial integration by parts give rise to boundary contributions. For these we note that
\begin{align*}
    \int_{\R  \times \R^{n+1}_+ } H_s v \partial_\ell v \,d\ell dy=-    \int_{\R  \times \R^{n+1}_+ } H_s v \partial_\ell v \,d\ell dy+\mbox{ (BT)},
\end{align*}
where (BT) denotes the boundary contributions from above and in particular,
\begin{align*}
\mbox{(BT)}
& = 2 \int\limits_{\R \times\R^n \times \{0\}} \left(\lim\limits_{y_{n+1}\rightarrow 0} y_{n+1}^{\frac{2s-1}{2}} \p_{\ell} v\right)  \left(\lim\limits_{y_{n+1}\rightarrow 0} y_{n+1}^{1-2s} \p_{n+1} (y_{n+1}^{\frac{2s-1}{2}}v)\right) dy' d\ell\\
& \quad -  2 \int\limits_{\R \times \R^n \times \{0\}} \left(\lim\limits_{y_{n+1}\rightarrow 0} y_{n+1}^{\frac{2s-1}{2}}  v\right) \left(\lim\limits_{y_{n+1}\rightarrow 0} y_{n+1}^{1-2s} \p_{n+1} (y_{n+1}^{\frac{2s-1}{2}}\p_{\ell} v)\right) dy' d\ell.
\end{align*}
Using the vanishing (weighted) Neumann boundary conditions together with the a priori regularity estimates from Lemma \ref{lem:reg_up_to_boundary}, we infer that the terms in (BT) vanish. 
Therefore, we also deduce that
\begin{align}\label{in carleman}
\int_{\R  \times \R^{n+1}_+ } H_s v \partial_\ell v \,d\ell dy=0.
\end{align}

Next, returning to \eqref{eq:expand_a} and inserting \eqref{in carleman} and the vanishing of the boundary data, we infer that
\begin{align*}
\begin{split}
\left\| L_h v \right\|_{L^2(\R  \times \R^{n+1}_+ )}^2
& = \|S v\|_{L^2(\R \times \R^{n+1}_+)}^2 + \|A v\|_{L^2(\R \times \R^{n+1}_+)}^2 +  \tau \int\limits_{\R\times \R^{n+1}_+ } h'' v^2 dy d\ell ,
\end{split}
\end{align*}
which implies that 
$$
\| L_h v \|_{L^2(\R  \times \R^{n+1}_+ )}^2\geq \tau \int\limits_{\R\times \R^{n+1}_+ } h'' v^2 dy d\ell.
$$
Finally, we plug $v=e^{\tau h(\ell)}w$ into the above inequality, which then yields the desired estimate. 
\end{proof}

\subsubsection{The Carleman estimate in Euclidean coordinates}

Relying on the previous discussion in parabolic conformal polar coordinates, we obtain a Carleman estimate in our original coordinates:

\begin{proposition}
\label{prop:Carl_eucl}
Let $s\in (0,1)$ and let $\tilde{u}\in \mathcal{W}([0,1]\times \overline{B_4^+})$ with $\supp(\tilde{u})\subset ((0,1)\times B_{4}^+) \setminus (0,0)$ be a weak solution to 
\begin{align*}
\begin{cases}
\left( x_{n+1}^{1-2s} \p_t + \nabla \cdot x_{n+1}^{1-2s} \nabla \right) \tilde{u}  = f & \mbox{ in } (0,1)\times B_4^+ , \\
\lim\limits_{x_{n+1} \rightarrow 0} x_{n+1}^{1-2s} \p_{n+1} \tilde{u} = 0 & \mbox{ on } (0,1)\times B_4',
\end{cases}
\end{align*}
where $f \in L^2([0,1]\times B_4^+, x_{n+1}^{1-2s})$.
Assume further that 
\begin{align*}
\phi(t,x):= - \frac{|x|^2}{8 t} + \tau h\left(- \frac{1}{4} \ln(t) \right)  ,
\end{align*}
with a convex function $h(\ell)$ which grows asymptotically linearly as $\ell \rightarrow \infty$. Then, there exists a constant $C>1$ such that for all $\tau \geq \tau_0>0$ we have
\begin{align*}
\tau \left\| e^{\phi} t^{ -\frac{1}{2}} (\bar{h}'')^{\frac{1}{2}} x_{n+1}^{\frac{1-2s}{2}} \tilde{u} \right\|_{L^2((0,\infty)\times \R^{n+1}_+)}^2
\leq C \left\| e^{\phi} t^{ \frac{1}{2}} x_{n+1}^{\frac{2s-1}{2}} f \right\|_{L^2((0,\infty) \times \R^{n+1}_+)}^2,
\end{align*}
for some constant $C>0$ independent of $\tilde u $ and $f$, where $\bar{h}''(t):=h''(r)|_{r=-{1\over 4}\ln(t)}$.
\end{proposition}

\begin{proof}
This follows directly from Proposition \ref{prop:Carl_conf} by setting $$w(\ell, y):= y_{n+1}^{\frac{1-2s}{2}}e^{-(n+2-2s)\ell} e^{-\frac{|y|^2}{2}} \tilde{u}(e^{-4 \ell}, 2 e^{-2 \ell} y)$$ and transforming back from parabolic conformal coordinates to Euclidean coordinates.
\end{proof}

\subsection{Global weak unique continuation}

In this section we deduce the weak unique continuation property from the Carleman estimate from Proposition \ref{prop:Carl_eucl}. In contrast to the results in the literature on unique continuation properties for fractional parabolic equations (see \cite{BG2017}), here we do \emph{not} assume that the equation $(\p_t - \D)^s u = Vu$ holds globally.

Our main aim in this section is to prove Theorem \ref{Thm UCP} in the following form:

\begin{proposition}
\label{prop:global_UCP}
Let $s\in (0,1)$, $n\in \N$ and $u \in \mathcal{H}^s(\R^{n+1})$. Assume that for some open set $\mathcal{U} \subset \R^n $ we have 
\begin{align}
\label{eq:overdet}
u \equiv 0, \ (\p_t - \D)^s u \equiv 0 \mbox{ in } (0,1) \times \mathcal{U}.
\end{align} 
Then, $u\equiv 0$ in $(0,1)\times \R^n$.
\end{proposition}

After having established the infinite vanishing order in Lemma \ref{lem:inf_van_order}, we now address the full weak unique continuation statement for which we still have to exclude super-polynomial decay towards the boundary. At this point we exploit the Carleman estimate from Proposition \ref{prop:Carl_eucl}.

\begin{proof}[Proof of Proposition \ref{prop:global_UCP}]
\emph{Step 1: Extension.} We first note that the vanishing property \eqref{eq:overdet} from above can be viewed in terms of its Caffarelli-Silvestre extension. Formulated in terms of this, we seek to show that if $u\in \mathcal{H}^s(\R \times \R^n)$ and if $\tilde{u}$ solves
\begin{align*}
\begin{cases}
\left(x_{n+1}^{1-2s}\p_t + \nabla \cdot x_{n+1}^{1-2s} \nabla \right) \tilde{u}  = 0 & \mbox{ in } (0,1) \times \R^{n+1}_+ ,\\
\tilde{u} = u & \mbox{ on } (0,1) \times \R^n \times \{0\},
\end{cases}
\end{align*} 
such that 
\begin{align*}
\tilde{u}=0 \mbox{ and } \lim\limits_{x_{n+1}\rightarrow 0} x_{n+1}^{1-2s}\p_{n+1} \tilde{u} = 0 \mbox{ in } (0,1) \times \mathcal{U} \times \{0\},
\end{align*}
then $\tilde{u}=0$ in $(0,1)\times \R^{n+1}_+$.
Using the result of Lemma \ref{lem:inf_van_order}, we infer that $\tilde{u}$ and all its (tangential, weighted normal and temporal) derivatives exist in a classical sense and vanish on strict subset of $(0,1) \times \mathcal{U} \times \{0\}$. As $\tilde{u}$ is smooth for every $x_{n+1}>0$ by parabolic regularity  and the infinite order of vanishing of $\tilde{u}$ up to $x_{n+1}=0$ (see Lemma \ref{lem:inf_van_order}), we hence obtain that $\tilde{u}$ is $C^{\infty}$ smooth up to the boundary $(0,1) \times \mathcal{U} \times \{0\}$. In particular, all integration by parts identities in the Carleman estimates are justified.\\

\emph{Step 2: Application of the Carleman estimate.} The vanishing of infinite order together with a cut-off argument allows us to apply the Carleman inequality from Proposition \ref{prop:Carl_eucl}. We discuss the details of this. Following \cite[Section 2]{KT09}, we set 
\begin{align*}
E_{\delta} &:= \left((0,2\delta^2) \times B_{2\delta}^+(0)\right)\setminus \left((0,\delta^{2}) \times B_{\delta}^+(0)\right),\\
F_{\tau}^{ext} & := \left(\left(0,\frac{2}{\tau}\right) \times B_{2}^+(0) \right)\setminus \left(\left(0,\frac{1}{\tau}\right) \times B_{1}^+(0) \right),\\
F_{\tau}^{int}&:= \left(\frac{1}{32\tau}, \frac{1}{16 \tau}\right) \times B_{\frac{1}{8}}^+(0) .
\end{align*}
Here we assume that $\delta\ll \tau^{-\frac{1}{2}}$. We now consider a cut-off function $\eta$ with the property that $\eta \equiv 1$ in $(0,1)\times B_{1}^+(0)$ and $\supp(\eta)\subset (0,2)\times B_{2}^+(0)$ and set
\begin{align*}
\tilde{u}_{\delta}(t,x) = \left( 1-\eta\left(\frac{t}{\delta^2},\frac{x}{\delta}\right) \right) \eta\left(\tau t, x \right) \tilde{u}(t,x).
\end{align*}
This function is admissible in the Carleman estimate from Proposition \ref{prop:Carl_eucl}; we use the weight function 
$$
\phi(t,x):= - \frac{|x|^2}{8 t} - \frac{1}{2} \ln(t) - \tau \ln(t) + \tau \frac{t}{3} = - \frac{|x|^2}{8t} + \tau h\left(-\frac{1}{4}\ln(t)\right),
$$ 
where in the notation from Proposition \ref{prop:Carl_eucl} we have $h(\ell) = 4 \ell - \frac{2}{\tau} \ell + \frac{1}{3} e^{- 4\ell}$ (which satisfies the requirements from Proposition \ref{prop:Carl_eucl}). 
By definition of $\tilde{u}_{\delta}$, we obtain
\begin{align*}
\left(x_{n+1}^{1-2s}\p_t + \nabla \cdot x_{n+1}^{1-2s} \nabla \right) \tilde{u}_{\delta} & = f
\end{align*}
with 
\begin{align*}
f(t,x)
&=
\tilde{u}(t,x)x_{n+1}^{1-2s} \p_t \left[ \left( 1-\eta\left(\frac{t}{\delta^2},\frac{x}{\delta}\right) \right) \eta\left(\tau t, x \right) \right]\\
& \quad + 2 x_{n+1}^{1-2s} \nabla \tilde{u}(t,x) \cdot \nabla \left[ \left( 1-\eta\left(\frac{t}{\delta^2},\frac{x}{\delta}\right) \right) \eta\left(\tau t,x \right) \right]\\
& \quad + \tilde{u}(t,x)\nabla \cdot x_{n+1}^{1-2s}\nabla \left[ \left( 1-\eta\left(\frac{t}{\delta^2},\frac{x}{\delta}\right) \right) \eta\left(\tau t, x \right) \right].
\end{align*}
We next seek to apply the Carleman estimate from Proposition \ref{prop:Carl_eucl} in order to deduce that $\tilde{u}_{\delta}=0$ in $\{0\}\times B_{1/8}^+$. To this end, we note that the contributions on the right hand side of the Carleman estimate are localized on the support of $f$, i.e. in the domains $E_{\delta}$ and $F_{\tau}^{ext}$. Thus, in these we seek to deduce upper bounds for the Carleman weight $e^{\phi}$. 

We begin with the bound in $E_{\delta}$.
Due to the infinite order of vanishing in $E_{\delta}$, it suffices to obtain a rough polynomial bound for the weight function there. 
We claim that for $\tau \geq \tau_0>1$ sufficiently large and $0<\delta$ sufficiently small, it holds
\begin{align}\label{estimate of phi in E delta}
\phi(t,x) \leq  - 3(\tau +1)\ln(\delta) + \tau \ln(\tau), \quad  \mbox{ for } (t,x)\in E_{\delta}.
\end{align}
In order to observe \eqref{estimate of phi in E delta}, we split the domain into two parts:
\begin{itemize}
\item[(a)]
In $(\delta^2, 2\delta^2)\times B_{2\delta}^+(0)$ we estimate
\begin{align*}
\phi(t,x)
\leq & - \frac{1}{2} \ln(\delta^2) - \tau \ln(\delta^2) + \tau {2\delta^2\over 3}
= -\left(\tau+\frac{1}{2}\right)\ln(\delta^2) + \tau {2\delta^2\over 3}\\
\leq &-2(\tau + 1)\ln(\delta).
\end{align*}
Here we used that $0<\delta ^2 <\delta \ll 1 \ll -\ln(\delta)$ for $\delta>0$ sufficiently small and that the first contribution in the definition of the weight $\phi(t,x)$ is always negative.
\item[(b)] In $(0,\delta^2) \times (B_{2\delta}^+(0)\setminus B_{\delta}(0))$ we estimate
\begin{align*}
\phi(t,x) \leq \widetilde{\psi}(t):= -\frac{\delta^2}{8t} - \frac{1}{2} \ln(t) - \tau \ln(t) + \tau \frac{\delta^2}{3}.
\end{align*}
We next maximize the auxiliary function $\widetilde{\psi}(t)$, which yields $t_{\max}= \frac{\delta^2}{8} \frac{1}{(\frac{1}{2}+\tau)}$. As a consequence,
\begin{align*}
\phi(t,x) &\leq \widetilde{\psi}(t_{\max},x)
\leq \frac{1}{2} + \tau - \frac{1}{2} \ln(\delta^2/(8\tau)) - \tau \ln(\delta^2/(8\tau)) + \tau \frac{2 \delta^2}{3}\\
&\leq \tau \ln(\tau) - 3(\tau + 1)\ln(\delta),
\end{align*}
if $\tau \geq \tau_0>1$ is sufficiently large and $\delta>0$ is sufficiently small. Combining above (a) and (b) yields \eqref{estimate of phi in E delta}.
\end{itemize}

In the domain $F_{\tau}^{ext}$ we obtain the upper bound
\begin{align}
\label{eq:Fext}
\phi(t,x) \leq \tau \ln(\tau) + \tau \ln(8) + \frac{1}{2} \ln(\tau)+ 4,
\end{align}
by the following observations.
Indeed, we split the domain $F^{ext}_{\tau}$ into two parts:
\begin{itemize}
\item[(a)] In $(1/\tau, 2/\tau) \times B_2^+$ we again drop the negative contributions and estimate as follows
\begin{align*}
\phi(t,x) \leq  - \frac{1}{2} \ln(1/\tau) - \tau \ln(1/\tau)  + \frac{2}{3}
= \tau \ln (\tau) + \frac{1}{2}\ln(\tau) + \frac{2}{3} .
\end{align*}
As this is dominated by the expression in \eqref{eq:Fext}, this implies the claim.
\item[(b)] In $(0,1/\tau) \times (B_2^+(0)\setminus B_1^+(0))$ we argue slightly more carefully. Here we first estimate
\begin{align*}
\phi(t,x) \leq \psi(t,x):= - \frac{1}{8t} - \frac{1}{2}\ln(t)- \tau \ln(t) + \frac{1}{3}.
\end{align*}
Next, for $\tau \geq \tau_0>1$ sufficiently large, we maximize the auxiliary function $\psi(t,x)$, which yields $t_{\max} = \frac{1}{8(1/2 + \tau)}$. Hence, we obtain 
\begin{align*}
\phi(t,x) \leq \psi(t_{\max},x) \leq \frac{1}{2}\ln(\tau) + \frac{1}{2}\ln(8)+ \tau \ln(\tau) + \tau\ln(8) + \frac{1}{3}.
\end{align*}
This also implies the claimed bound \eqref{eq:Fext}.
\end{itemize} 

As a consequence, we bound the right hand side of the Carleman estimate as follows:
\begin{align*}
& \quad  \left\|e^{\phi} t^{\frac{1}{2}} x_{n+1}^{\frac{2s-1}{2}} f \right\|_{L^2((0,\infty) \times \R^{n+1}_+)} \\
&\leq \left\|e^{\phi} t^{\frac{1}{2}} x_{n+1}^{\frac{2s-1}{2}} f\right\|_{L^2(E_{\delta})} + \left\|e^{\phi} t^{\frac{1}{2}}x_{n+1}^{\frac{2s-1}{2}} f\right\|_{L^2(F^{ext}_{\tau})}\\
&\leq C\delta^{-3(\tau+1)}e^{\tau \ln(\tau)}  \left(\left\|x_{n+1}^{\frac{1-2s}{2}} \tilde{u}\right\|_{L^2(E_{\delta})}
+  \left\|x_{n+1}^{\frac{1-2s}{2}} \nabla \tilde{u}\right\|_{L^2(E_{\delta})} \right) \\
& \quad + Ce^{\tau \ln(\tau) + \tau \ln(8) + \frac{1}{2}\ln(\tau) +4}  \left(\left\|x_{n+1}^{\frac{1-2s}{2}} \tilde{u}\right\|_{L^2(F_{\tau}^{ext})}
+  \left\|x_{n+1}^{\frac{1-2s}{2}} \nabla \tilde{u}\right\|_{L^2(F_{\tau}^{ext})} \right).
\end{align*}
Using Caccioppoli's estimate together with the vanishing Dirichlet and Neumann boundary conditions to bound the gradient terms, we infer
\begin{align}
\label{eq:Carl_upper}
\begin{split}
\left\|e^{\phi} t^{\frac{1}{2}} x_{n+1}^{\frac{2s-1}{2}}f\right\|_{L^2((0,\infty) \times \R^{n+1}_+)}
&\leq C \delta^{-3(\tau +1)} e^{\tau \ln(\tau)} \left\| x_{n+1}^{\frac{1-2s}{2}} \tilde{u}\right\|_{L^2(\widetilde{E}_{\delta})}\\
& \quad + C e^{\tau \ln(\tau) + \tau \ln(8) + \frac{1}{2}\ln(\tau) +4}
\left\|x_{n+1}^{\frac{1-2s}{2}}\tilde{u}\right\|_{L^2(\widetilde{F}^{ext}_{\tau})},
\end{split}
\end{align}
where 
\begin{align*}
\tilde{E}_{\delta} &:= ((0,3\delta^2) \times B_{3\delta}^+(0))\setminus \left(\left(0,\frac{\delta^2}{2}\right) \times B_{\frac{\delta}{2}}^+(0) \right),\\
\widetilde{F}_{\tau}^{ext} & := \left(\left(0,\frac{3}{\tau}\right) \times B_{3}^+(0) \right)\setminus \left(\left(0,\frac{1}{2\tau}\right) \times B_{\frac{1}{2}}^+(0) \right).
\end{align*}

In order to infer the desired unique continuation result, it hence remains to bound $\phi(t,x)$ in $F_{\tau}^{int}$ from below. In this region we have the following lower bound on the weight function
	\begin{align*}
\phi(t,x)
\geq -\frac{\tau}{32} + \tau\ln(\tau)+ \tau\ln(16)+ \frac{1}{2}\ln(\tau) .
\end{align*}
Indeed, as in case (b) of the discussion of the estimate in $F^{ext}_{\tau}$ from above we estimate
\begin{align*}
\phi(t,x) \geq - \frac{1}{512 t} - \frac{1}{2}\ln(t)-\tau \ln(t)  =: \bar{\psi}(t,x).
\end{align*}
Noting that the critical point of $\bar{\psi}(t,x)$ is at $t_{crit}=\frac{1}{512(\tau + 1/2)}\notin \left( \frac{1}{32\tau }, \frac{1}{16\tau } \right)$ and computing the sign of $\bar{\psi}'$ in $ \left( \frac{1}{32\tau }, \frac{1}{16\tau } \right)$, we observe that $\psi$ is monotone decreasing in $ \left( \frac{1}{32\tau }, \frac{1}{16\tau } \right)$. Thus, we obtain
\begin{align*}
\phi(t,x) \geq \bar{\psi}\left(\frac{1}{16\tau},x\right) =  -\frac{\tau}{32} + \frac{1}{2} \ln(\tau) + \frac{1}{2}\ln(16) + \tau \ln(\tau) + \tau\ln(16).
\end{align*}
Hence, the left hand side of the Carleman estimate from Proposition \ref{prop:Carl_eucl} can be bounded from below by
\begin{align}
\label{eq:Carl_lower}
\left\|e^{\phi}t^{-\frac{1}{2}} (\bar{h}'')^{\frac{1}{2}} x_{n+1}^{\frac{1-2s}{2}} \tilde{u}_{\delta}\right\|_{L^2((0,\infty)\times \R^{n+1}_+)} 
\geq e^{\tau\ln(\tau) + \tau \left(\ln(16)-\frac{1}{32}\right)}\left\|x_{n+1}^{\frac{1-2s}{2}} \tilde{u}\right\|_{L^2(F^{int}_{\tau})}.
\end{align}
Combining the bounds from \eqref{eq:Carl_upper} and \eqref{eq:Carl_lower} (using that $\ln(16)-\frac{1}{32}\geq \ln(14)$ and $\tau \ln(8) + \frac{1}{2}\ln(\tau)\leq \tau \ln(9)$ for $\tau\geq \tau_0>1$ sufficiently large), for $\tau>\tau_0>1$ sufficiently large, we thus infer that
\begin{align*}
& e^{ \tau (\ln(\tau) + \ln(14))}
\left\|x_{n+1}^{\frac{1-2s}{2}} \tilde{u}\right\|_{L^2(F^{int}_{\tau})} \\
\leq &C\left( e^{\tau(\ln(\tau)+\ln(9))} \left\| x_{n+1}^{\frac{1-2s}{2}} \tilde{u}\right\|_{L^2(\widetilde{F}^{ext}_{\tau})} + \delta^{-4\tau} e^{\tau \ln(\tau)}\left\|x_{n+1}^{\frac{1-2s}{2}}\tilde{u}\right\|_{L^2(\widetilde{E}_{\delta})} \right).
\end{align*}
Using the infinite order of vanishing of $\tilde{u}$ (see Lemma \ref{lem:inf_van_order}), we may pass to the limit $\delta \rightarrow 0$ (for fixed $\tau>\tau_0>1$). As a consequence,
\begin{align*}
\left\|x_{n+1}^{\frac{1-2s}{2}} \tilde{u}\right\|_{L^2(F^{int}_{\tau})}
\leq C e^{-\tau \ln(14/9)}\left\|x_{n+1}^{\frac{1-2s}{2}} \tilde{u}\right\|_{L^2(\widetilde{F}^{ext}_{\tau})}.
\end{align*}
By the local $C^{0,\alpha}$ regularity of $\tilde{u}$ (see Lemma \ref{lem:reg_up_to_boundary}) and  the mean value theorem (of integral form), there exists $\tilde{\tau} \in \left( \frac{1}{32 \tau}, \frac{1}{16 \tau} \right)$ such that 
\begin{align*}
\left\|x_{n+1}^{\frac{1-2s}{2}} \tilde{u}(\tilde{\tau},\cdot, \cdot)\right\|_{L^2(B_{1\over 8}^+)} 
= 32\tau \left\|x_{n+1}^{\frac{1-2s}{2}} \tilde{u}\right\|_{L^2(F^{int}_{\tau})}.
\end{align*}
Choosing $\tau = \frac{1}{t}$ and passing to the limit $t \rightarrow 0$ (while using the regularity of $\tilde{u}$), we obtain
\begin{align*}
\left\|x_{n+1}^{\frac{1-2s}{2}} \tilde{u}\right\|_{L^2(B_{\frac{1}{8}}^+)} \leq C 32\lim\limits_{t \rightarrow 0} \frac{1}{t} e^{-\ln(14/9)\frac{1}{t}} \left\|x_{n+1}^{\frac{1-2s}{2}} \tilde{u}\right\|_{L^2(\widetilde{F}^{ext}_{1/t})} =0,
\end{align*} 
whence we conclude that $\tilde{u}(0,x)=0$ for $x\in B_{\frac{1}{8}}^+$. 

Since this holds for all the time slices on which $\tilde{u}(t,x) = 0 $ in $(0,1)\times B_{1\over 8}'$, we obtain that $\tilde{u} \equiv 0$ in $(0,1)\times B_{1\over 8}^+$ by using the time and spatial tangential translation invariance of the operator $x_{n+1}^{1-2s}\p_t - \nabla \cdot x_{n+1}^{1-2s} \nabla $. As a consequence of spatial unique continuation in the upper half-plane, this then entails that $u\equiv 0$ in $(0,1)\times \R^n $, which yields the desired result. 
\end{proof}

\begin{remark}\label{remark for strong unique continuation for adjoint}
The global weak unique continuation property also holds for the adjoint fractional parabolic operator $\LL^s_\ast$. In other words, suppose that $u=\LL^s_\ast u =0$ in $(0,1)\times \mathcal{U}$ for some nonempty open set $\mathcal{U}\subset \R^n$, then $u\equiv 0$ in $(0,1)\times \R^n$. The proof follows along the same lines as the proof of Proposition \ref{prop:global_UCP} by invoking the Carleman estimate derived in Section \ref{Section 5}.
\end{remark}

%%%%%%%%%%%%%%%%%%%%%%%%%%%%%%%%%%%%%%%%%%
%%%%%%  Runge Approximation  %%%%%%%%%%%%%
%%%%%%%%%%%%%%%%%%%%%%%%%%%%%%%%%%%%%%%%%%

\section{Runge approximation and the proof of main theorems}\label{Section 6}

Recall that the initial exterior value problem of the fractional parabolic equation is given by \eqref{noncal_heat} with zero initial value. As explaind in Proposition \ref{prop:cut-off}, one can multiply a cutoff function to ensure that the future data are zero, without changing the solution in a given (time-space) domain. Therefore, it suffices to consider  Runge approximation results in these time space domains.

\subsection{Runge approximation}
For $0<s<1$ and $T>0$, we recall the notation $\Omega_T=(-T,T)\times \Omega \subset \R^{n+1}$. Let $Q\in L^\infty(\Omega_T)$ satisfy the eigenvalue condition \eqref{eigenvalue condition} and $u=u_f\in \HH^s(\R^{n+1})$ be a solution of 
\begin{align}\label{equ in Runge}
(\LL^s+Q)u_f=0 \text{ in }\Omega_T, \quad \text{ with }\quad u_f=f \text{ in } (-T,T)\times \Omega_e, \text{ and }u_f=0 \text{ for }t\leq -T. 
\end{align}
Then $\chi_{(-\infty,T]}(t)u_f(t,x)$ is the unique solution of \eqref{equ in Runge}.

\begin{lemma}[Runge approximation] \label{Lemma Runge approximation}
	For $n\geq 1$, let $\mathcal U \subset \Omega_e$ be an open subset and $T>0$ be a real number. Then the set 
	$$
	\mathcal{R}=\left\{u_f|_{\Omega_T }:\ u_f\ \hbox{ the solution to \eqref{equ in Runge}},\ f\in C^\infty_c((-T,T)\times \mathcal U) \right\}
	$$
	is dense in $L^2(\Omega_T)$.
\end{lemma}

\begin{proof}
	The proof is similar to the proof of \cite[Theorem 1.2]{ghosh2017calderon} and \cite[Theorem 1.2]{ghosh2016calder}. Invoking the Hahn-Banach theorem, it is sufficient to show that if $(v, w)_{L^2(\Omega_T)}=0$ for all $v\in \mathcal R$, then necessarily $w\equiv 0$.
	Let thus $w$ be as described above, i.e. let us assume that	 
	\begin{align*}
	\left(\chi_{(-\infty,T]}u_f,w\right)_{L^2(\Omega_T)}=\left(u_f,w\right)_{L^2(\Omega_T)}=0, \quad \text{ for all }f\in C^\infty_c((-T,T)\times \mathcal U),
	\end{align*}
	where $\chi_{(-\infty,T]}u_f$ be the unique solution of \eqref{equ in Runge} in $\Omega_T$. Here we have utilized the future data will not affect the solution in $\Omega_T$ (see Section \ref{Section 3}).
	Next, let $\phi\in \HH^s(\R^{n+1})$ be the solution of 
	\begin{align}\label{equation of phi}
	\begin{cases}
	(-\p_t -\Delta)^s\phi +Q \phi=w & \hbox{in } (-T,T)\times \Omega,\\
	\phi =0 & \hbox{in } (-T,T) \times \Omega_{e}) \cup (-\infty, -T] \times \R^n) \cup ([T, \infty)\times \R^n).
	\end{cases}
	\end{align}
Then,
	\begin{align*}
	(u_f, w)_{L^2((-T,T)\times\Omega)} &= (u_f -f,   (-\p_t -\Delta)^s\phi +Q \phi)_{L^2((-T,T)\times \R^n)}\\
	& = -(f,  (-\p_t -\Delta)^s\phi )_{L^2((-T,T)\times \mathcal U) },
	\end{align*}	
	for all $f\in C^\infty_c((-T,T)\times \mathcal U)$, where in the last identity we used the fact that $f$ is supported in $(-T,T)\times \mathcal U$. 
	Thus, we arrive at
	$$
	(-\p_t -\Delta)^s\phi = 0\qquad \ \hbox{and}\qquad \ \phi=0\ \ \hbox{in }(-T,T)\times \mathcal U.
	$$
Then, by virtue of the global weak unique continuation property (see Remark \ref{remark for strong unique continuation for adjoint}), we obtain
	$$
	\phi=0 \ \ \hbox{in }(-T,T)\times \R^n.
	$$
	Combining this with the exterior condition of $\phi=0$ in past and future time from \eqref{equation of phi}, we obtain that $\phi(t,x)\equiv 0 $ for all $(t,x)\in \R^{n+1}$, from which we infer that $\LL^s \phi=0$ in $\R^{n+1}$. 	
 	Thus, recalling the equation \eqref{equation of phi} again, we infer that $w \equiv 0$.
\end{proof}

\begin{remark}
	By similar arguments, one can also obtain the Runge approximation property for the adjoint fractional parabolic equation. More specifically, for $0<s<1$ and $T>0$, let $Q\in L^\infty(\Omega_T)$ satisfy the eigenvalue condition \eqref{eigenvalue condition} and let $v=v_g\in \HH^s(\R^{n+1})$ be a solution of 
	$$
	(\LL^s_\ast +Q)v_g=0 \text{ in }\Omega_T, \quad \text{ with }\quad v_g=g \text{ in } (-T,T)\times \Omega_e,
	$$ 
	and $v_g = 0$ for $t\geq T$ (see Section \ref{Section 3} for details). Then $\chi_{[-T,\infty)}u_g$ is the unique solution of the above equation in $\Omega_T$. 
	Let $\mathcal U \subset \Omega_e$ be an open subset, then the set 
	$$
	\left\{v_g|_{\Omega_T }:\ g\in C^\infty_c((-T,T)\times \mathcal U) \right\}
	$$
	is dense in $L^2(\Omega_T)$. This result follows directly from the proof of Lemma \ref{Lemma Runge approximation} and a corresponding variant Theorem \ref{Thm UCP}, therefore, we omit the details here.
\end{remark}

\subsection{Proof of Theorem \ref{MAIN THEOREM}}

With the help of the Runge approximation in Lemma \ref{Lemma Runge approximation}, now we can address the global uniqueness result for the fractional parabolic equation. The proof is similar to that in \cite{ghosh2017calderon} and \cite{ghosh2016calder}.
 
\begin{proof}[Proof of Theorem \ref{MAIN THEOREM}]
	Suppose that  $\Lambda_{Q_{1}}f|_{(-T,T)\times \mathcal{U}_{2}}=\Lambda_{Q_{2}}f|_{(-T,T)\times\mathcal{U}_{2}}$
	for any $f\in C_{c}^{\infty}((-T,T)\times \mathcal{U}_{1})$, where $\mathcal{U}_{1}$
	and $\mathcal{U}_{2}$ are arbitrary open subsets of $\Omega_e=\R^n\setminus\overline{\Omega}$. By utilizing the integral
	identity in Lemma \ref{Lem Integral identitiy}, we have 
	\begin{align}
	\label{eq:Aless}
	\int_{\Omega_T}(Q_{1}-Q_{2})u_{1}u_{2}dxdt=0,
	\end{align}
	where $u_{1},u_{2}\in \HH^{s}(\mathbb{R}^{n+1})$ are the solutions of  
	\[
	(\mathcal{L}^{s}+Q_{1})u_{1}=0 \text{ in }\Omega_T \text{ with }u_1 =0 \text{ for } \{t\leq -T\},
	\]
	and 
	\[
	(\mathcal{L}_{\ast}^{s}+Q_{2})u_{2}=0  \text{ in }\Omega_T\text{ with } u_2 =0 \text{ for }\{t\geq T\}. 
    \]
    Here $u_{1}$ and $u_{2}$ have the same exterior values $f_{j}\in C_{c}^{\infty}((-T,T)\times \mathcal{U}_{j})$,
	for $j=1,2$.
	
	Given any function $g\in L^{2}(\Omega_T)$, and using the Runge approximation result (see Lemma \ref{Lemma Runge approximation}), it is possible to find two sequences $ \{u_j^{(1)} \}_{j\in \mathbb N}$, $\{u_j^{(2)} \}_{j\in \mathbb N}$ of functions
	in $\HH^{s}(\mathbb{R}^{n+1})$ that satisfy 
	\begin{align*}
	& (\mathcal{L}^{s}+Q_{1})u_{j}^{(1)}=({\mathcal{L}^{s}_{\ast}}+Q_{2})u_{j}^{(2)}=0\text{ in }\Omega_T,\\
	& \mbox{supp}(u_{j}^{(1)})\subseteq\overline{ \Omega^{(1)}_T}\quad \mbox{ and }\quad\mbox{supp}(u_{j}^{(2)})\subseteq \overline{\Omega^{(2)}_T},\\
	& u_{j}^{(1)}|_{\Omega_T}=g+r_{j}^{(1)},\qquad  u_{j}^{(2)}|_{\Omega_T}=1+r_{j}^{(2)},
	\end{align*}
	where $\Omega^{(1)}_T$, $\Omega^{(2)}_T$ are two open sets in $\mathbb{R}^{n+1}$ which contain $\Omega_T$, and $r_{j}^{(1)},r_{j}^{(2)}\to0$ in $L^{2}(\Omega_T)$
	as $j\to\infty$. Inserting these solutions into the integral identity \eqref{eq:Aless}
	and taking $j\to\infty$, then leads to
	\[
	\int_{\Omega_T}(Q_{1}-Q_{2})g\, dxdt=0.
	\]
	Since $g\in L^{2}(\Omega_T)$ is arbitrary, we must hence have $Q_1=Q_2$ in $\Omega_T$, which completes the proof.
\end{proof}

\subsection{Single measurement results -- proof of Theorem \ref{thm:single_measurement}}
\label{sec:single_meas}

As in \cite{GRSU18}, it is possible to exploit the global weak unique continuation property together with a Tikhonov regularization argument to infer a single measurement recovery result.

\begin{proposition}[Tikhonov regularization]
\label{prop:Tikh}Assume that $W\subset \R^n$ is open with $\overline{\Omega}\cap\overline{W}=\emptyset$.
Let $s\in (0,1)$, $u\in \widetilde{\mathcal{H}}^s((-T,T)\times \Omega)$ and let $h= (\p_t-\D)^s u$. Then,  $u = \lim\limits_{\alpha \rightarrow 0} u_{\alpha}$  in $\widetilde{\mathcal{H}}^s((-T,T)\times \Omega)$, where for $\alpha \in (0,1)$, the functions $u_{\alpha}$ are defined as 
\begin{align*}
u_{\alpha} = \argmin_{v \in \widetilde{\mathcal{H}}^s((-T,T)\times \Omega)} \left( \|(\p_t - \D)^s v- h\|_{\mathcal{H}^{-s}((-T,T) \times W)}^2 + \alpha \|v\|_{\widetilde{\mathcal{H}}^s((-T,T)\times \Omega)} ^2\right).
\end{align*}
\end{proposition}

\begin{proof}
By possibly shrinking $W$, without loss of generality we may assume that $W$ is a bounded open set such that $\overline{\Omega}\cap \overline{W}=\emptyset$. As a consequence, we obtain the compactness of the mapping 
\begin{align*}
L: \widetilde{\mathcal{H}}^{s}((-T,T)\times \Omega) \rightarrow
\mathcal{\mathcal{H}}^{-s}((-T,T)\times W),\
  v \mapsto \left. (\p_t - \D)^s v \right|_{(-T,T)\times W}.
\end{align*}
Here the compactness follows from the pseudolocality of the operator $(\p_t - \D)^s$: Setting 
\begin{itemize}
\item $\chi_1$ to be a smooth cut-off function which is equal to one in a small neighbourhood of $(-T,T)\times W$ and vanishes outside a slightly larger neighbhourhood of the same set,
\item and $\chi_2$ a smooth cut-off function which is equal to one in a small neighbourhood of $(-T,T) \times \Omega$, which vanishes in a slighly larger neighbourhood of this set and which is constructed such that the closures of the support of $\chi_1$ and $\chi_2$ are empty,
\end{itemize}
the pseudolocality of the operator $(\p_t -\D)^s$ implies that $L = \chi_2 (\p_t -\D)^s \chi_1$ is a compact operator. With this observation it is possible to invoke the general theory of Tikhonov regularization, for which we refer for instance to \cite[Chapter 4]{ColtonKress}.
\end{proof}

\begin{remark}
An alternative argument yielding the compactness of the operator $L$ is to invoke the representation formula from \eqref{eq:integral_rep}.
\end{remark}

With this in hand, we proceed to the proof of the single measurement result, which as in the static case relies on the combination of Proposition \ref{prop:Tikh} with the weak unique continuation property of Proposition \ref{prop:global_UCP}.

\begin{proof}[Proof of Theorem \ref{thm:single_measurement}]
We consider a splitting of $u$ into a function $v\in \mathcal{H}^s((-T,T)\times \Omega)$ and the boundary data {$f\in \widetilde{\mathcal{H}}^s((-T,T)\times \Omega_e)$}, i.e. $u= v + f$. By construction and by the representation of the Dirichlet-to-Neumann map, we have (as {$ \widetilde{\mathcal{H}}^s((-T,T)\times \Omega_e)$} functions)
\begin{align*}
(\p_t - \D)^s v 
= (\p_t - \D)^s u - (\p_t - \D)^s f
= \Lambda_Q f - (\p_t - \D)^s f.
\end{align*}
As $\Lambda_Q f - (\p_t - \D)^s f$ is known, we can apply Proposition \ref{prop:Tikh} to reconstruct $v$ globally and constructively from this. 
Returning to $u= f+ v$ implies that as $f$ and $v$ are known globally, also $u$ is known globally. As a consequence, it is possible to solve the equation satisfied by $u$ for the potential $Q$ (we recall the regularity result from Lemma \ref{lem:reg_up_to_boundary} which imply that all involved quantities in the quotient exist in a pointwise sense):
\begin{align*}
Q(t,x) = -\frac{(\p_t - \D)^s u(t,x)}{u(t,x)}.
\end{align*}
Using the weak unique continuation property of Proposition \ref{prop:global_UCP} as well as the regularity of the potential $Q(t,x)$, we then conclude that for any $(t,x)\in (-T,T)\times \Omega$, there exists a sequence $(t_k, x_k) \in (-T,T)\times \Omega$ such that $u(t_k, x_k) \neq 0$ and $(t_k, x_k) \rightarrow (t,x)$. Therefore, for every $(t,x)\in (-T,T)\times \Omega$, by continuity,	 we have
\begin{align*}
Q(t,x) = \lim\limits_{k \rightarrow \infty} Q(t_k, x_k)
= -\lim\limits_{k \rightarrow \infty} \frac{(\p_t - \D)^s u(t_k, x_k)}{u(t_k, x_k)}.
\end{align*}
This allows us to constructively recover $Q(t,x)$ for $(t,x)\in (-T,T)\times \Omega$.
\end{proof}

\vskip1cm
\textbf{Acknowledgment.}
R.-Y. Lai is partially supported by NSF grant DMS-1714490. Y.-H. Lin is supported by the Finnish Centre of Excellence in Inverse Modelling and Imaging (Academy of Finland grant 284715) and also by the Academy of Finland project number 309963. 

\bibliographystyle{abbrv}
\bibliography{ref}

\end{document}